\documentclass[11pt]{article}
\usepackage[utf8]{inputenc}
\usepackage{amsmath}
\usepackage[hmargin=1in,vmargin=1in]{geometry}
\usepackage{graphics}
\usepackage[toc,page]{appendix}
\usepackage{amsfonts}
\usepackage{MnSymbol}
\usepackage{wasysym}
\usepackage{mathtools}
\usepackage{enumerate}
\usepackage{amsthm}
\usepackage{graphicx}
\usepackage{float}
\usepackage{chngcntr}
\usepackage{thmtools, thm-restate}
\theoremstyle{plain}
\newtheorem{thm}{Theorem}[section]
\newtheorem{lemma}[thm]{Lemma}

\newtheorem{prop}[thm]{Proposition}

\theoremstyle{definition}
\newtheorem{defn}[thm]{Definition}

\newtheorem{remark}[thm]{Remark}

\numberwithin{equation}{section}
\counterwithin{figure}{section}
\title{The First Moment of $L(\frac{1}{2},\chi)$ for Real Quadratic Function Fields}
\author{J.C. Andrade \\J. MacMillan}
\date{}
\newcommand\blfootnote[1]{%
  \begingroup
  \renewcommand\thefootnote{}\footnote{#1}%
  \addtocounter{footnote}{-1}%
  \endgroup
}
\begin{document}
\maketitle
\blfootnote{Date: August 09, 2019}
\par\noindent
ABSTRACT: In this paper we use techniques first introduced by Florea to improve the asymptotic formula for the first moment of the quadratic Dirichlet L-functions over the rational function field, running over all monic, square-free polynomials of even degree at the central point. With some extra technical difficulties that doesn't appear in Florea's paper, we prove that there are extra main terms of size $(2g+2)q^{\frac{2g+2}{3}}, q^{\frac{g}{6}+\left[\frac{g}{2}\right]}$ and $q^{\frac{g}{6}+\left[\frac{g-1}{2}\right]}$, whilst bounding the error term by $q^{\frac{g}{2}(1+\epsilon)}$.

\section{Introduction}
An important and well-studied problem in analytic number theory is to understand the asymptotic behaviour of moments of families of L-functions. Considering the family of Dirichlet L-functions, $L(s,\chi_d)$, with $\chi_d$ a real primitive Dirichlet character modulo $d$ defined by the Jacobi symbol $\chi_d(n)=\left(\frac{d}{n}\right)$, a problem is to understand the asymptotic behaviour of 
\begin{equation}\label{eq:1.1}
    \sum_{0<d\leq D}L\left(s,\chi_d\right)^k,
\end{equation}
summing over fundamental discriminants $d$, as $D\rightarrow\infty$. In this context, Jutila, \cite{Jutila1981}, proved, when $s=\frac{1}{2}$, that 
\begin{equation}\label{eq:1.2}
    \sum_{0<d\leq D}L\left(\frac{1}{2},\chi_d\right)=\frac{P(1)}{4\zeta(2)}\left[\log\left(\frac{D}{\pi}\right)+\frac{\Gamma'}{\Gamma}\left(\frac{1}{4}\right)+4\gamma-1 +4\frac{P'}{P}(1)\right]+O(D^{\frac{3}{4}+\epsilon}),
\end{equation}
where $\epsilon>0$ and 
\begin{equation*}
    P(s)=\prod_P\left(1-\frac{1}{p^s(p+1)}\right).
\end{equation*}
Goldfeld and Hoffstein, \cite{Goldfeld1985}, improved the error term to $D^{\frac{19}{32}+\epsilon}$. Young, \cite{Young2009THEL-FUNCTIONS}, showed that the error term is bounded by $D^{\frac{1}{2}+\epsilon}$ when considering the smoothed first moment. Jutila, \cite{Jutila1981}, computed the second moment and Soundararajan, \cite{Soundararajan2000NonvanishingS=1/2}, computed the second and third moments, when averaging over real, primitive, even characters with conductors $8d$. It is conjectured that 
\begin{equation*}
    \sum_{0<d\leq D}L\left(\frac{1}{2},\chi_d\right)^k\sim C_kD(\log D)^{\frac{k(k+1)}{2}},
\end{equation*}
where the sum is over fundamental discriminants. Keating and Snaith, \cite{Keating2000Random1/2}, conjectured a precise value for $C_k$ and Conrey \textit{et al}, \cite{Conrey2005IntegralL-functions}, conjectured the integral moments and formulas for the principal lower order terms. \\
\par\noindent
In the function field setting, the analogue problem is to understand the asymptotic behaviour of 
\begin{equation}\label{eq:1.3}
    \sum_{D\in\mathcal{H}_{2g+1}}L(s,\chi_D)^k
\end{equation}
as $|D|=q^{\text{deg}(D)}\rightarrow\infty$, where $\mathcal{H}_{2g+1}$ denotes the space of monic, square-free polynomials of degree $2g+1$ over $\mathbb{F}_q[x]$, which corresponds to the imaginary quadratic function field, and $L(s,\chi_D)$ denotes the quadratic Dirichlet L-function for the rational function field. Since we are letting $|D|\rightarrow\infty$, there are two limits to consider. The first is to fix $g$ and let $q\rightarrow\infty$ and the second is to fix $q$ and let $g\rightarrow \infty$. Katz and Sarnak, \cite{Katz1999,Katz1999ZerosSymmetry} used equidistribution results to relate the $q$ limit of (\ref{eq:1.3}) to a random matrix theory integral, which was then computed by Keating and Snaith, \cite{Keating2000Random1/2}. Therefore we will concentrate on the other limit, namely when $q$ is fixed and we let $g\rightarrow\infty$. In this context, Andrade and Keating, \cite{Andrade2012}, computed the first moment of (\ref{eq:1.3}), when $s=\frac{1}{2}$, which is considered to be the function field analogue of Jutila's result (\ref{eq:1.2}). In particular they proved the following result. 
\begin{thm}[Andrade and Keating]
Let $q$ be the fixed cardinality of the ground field and assume for simplicity that $q\equiv 1 (\text{mod }4)$. Then 
\begin{equation}\label{eq:1.4}
    \sum_{D\in\mathcal{H}_{2g+1}}L\left(\frac{1}{2},\chi_D\right)=\frac{P(1)}{2\zeta_{\mathbb{A}}(2)}|D|\left[\log_q|D|+1+\frac{4}{\log q}\frac{P'}{P}(1)\right]+O(|D|^{\frac{3}{4}+\frac{\log_q2}{2}}),
\end{equation}
where 
\begin{equation*}
    P(s)=\prod_P\left(1-\frac{1}{|P|^s(|P|+1)}\right)
\end{equation*}
and $\zeta_{\mathbb{A}}$ denotes the zeta function associated with $\mathbb{A}=\mathbb{F}_q[x]$. 
\end{thm}
\par\noindent
Andrade and Keating, \cite{Andrade2014}, also conjectured asymptotic formulas for higher and integral moments for (\ref{eq:1.3}) which is considered to be the function field analogue of Keating and Snaith's result, \cite{Keating2000Random1/2} and Conrey \textit{et al}, \cite{Conrey2005IntegralL-functions}. Rubinstein and Wu, \cite{Rubinstein2015}, provided numerical evidence for the conjecture given by Andrade and Keating, \cite{Andrade2014}. They numerically computed the moments for $k\leq 10$ and $d\leq 18$, where $d=2g+1$, for various values of $q$ and compared them to the conjectured formulas. \\
\par\noindent
Using a similar method to Young's, \cite{Young2009THEL-FUNCTIONS}, in the number field case, Florea, \cite{Florea2017}, improved the asymptotic formula obtained by Andrade and Keating (\ref{eq:1.4}) by obtaining a secondary main term of size $gq^{\frac{2g+1}{3}}$ and bounding the error term by $q^{\frac{g}{2}(1+\epsilon)}$. 
\begin{thm}[Florea]
    Let $q$ be a prime with $q\equiv 1(\text{mod }4)$. Then
    \begin{equation}
        \sum_{D\in\mathcal{H}_{2g+1}}L\left(\frac{1}{2},\chi_D\right)=\frac{P(1)}{2\zeta_{\mathbb{A}}(2)}q^{2g+1}\left[(2g+1)+1+\frac{4}{\log q}\frac{P'}{P}(1)\right]+q^{\frac{2g+1}{3}}R(2g+1)+O(q^{\frac{g}{2}(1+\epsilon)}),
    \end{equation}
where $R$ is a polynomial of degree 1 that can be explicitly calculated. 
\end{thm}
\par\noindent
Using a similar method, Florea, \cite{Florea2017a,Florea2017c}, computed the second, third and fourth moments of (\ref{eq:1.3}) at $s=\frac{1}{2}$ and showed that these asymptotic formulas agree with the formulas conjectured by Andrade and Keating in \cite{Andrade2014}.\\
\par\noindent
In \cite{Andrade2012a}, Andrade obtained an asymptotic formula for the first moment of $(\ref{eq:1.3})$ at $s=1$. In particular, he proved that
\begin{equation}\label{eq:1.6}
    \sum_{D\in\mathcal{H}_{2g+1}}L(1,\chi_D)=|D|P(2)+O((2q)^g). 
\end{equation}
Using the techniques presented by Florea, Andrade and Jung, \cite{Andrade2018}, improved the asymptotic formula, (\ref{eq:1.6}), by obtaining a secondary main term of size $q^{\frac{g}{3}}$ and bounding the error term by $q^{g\epsilon}$, for any $\epsilon>0$. In particular, they proved that
\begin{equation}
    \sum_{D\in\mathcal{H}_{2g+1}}L(1,\chi_D)=P(2)q^{2g+1}+c_1q^{\frac{g}{3}}+O(q^{g\epsilon}),
\end{equation}
where $c_1$ is a constant that can be explicitly calculated. \\
\par\noindent
In a recent paper, Bae and Jung, \cite{Bae2019}, improved the asymptotic formula for the second derivative of (\ref{eq:1.3}) at $s=\frac{1}{2}$, that was obtained by Andrade and  Rajagopal, \cite{Andrade2016MeanI}, using the techniques presented by Florea. In particular, compared to the asymptotic formula obtained by Andrade and  Rajagopal, Bae and Jung were able to obtain a secondary main term of size $g^3q^{\frac{2g+1}{3}}$ whilst also bounding the error term by $q^{\frac{g}{2}(1+\epsilon)}.$\\
\par\noindent
Another problem in function fields is to understand the aymptotic behaviour of 
\begin{equation}\label{eq:1.8}
    \sum_{D\in\mathcal{H}_{2g+2}}L(s,\chi_D)^k,
\end{equation}
as $|D|\rightarrow\infty$, where $\mathcal{H}_{2g+2}$ denotes the space of monic, square-free polynomials of degree $2g+2$, which corresponds to the real quadratic function field. In particular we concentrate on when $q$ is fixed and letting $g\rightarrow\infty$. In this context, Jung, \cite{Jung2013NoteEnsemble}, obtained an asymptotic formula for the first moment of (\ref{eq:1.8}) at $s=\frac{1}{2}$. 
\begin{thm}[Jung]
    Assume that $q$ is odd and greater than 3. Then we have
    \begin{equation}\label{eq:1.9}
        \sum_{D\in\mathcal{H}_{2g+2}}L\left(\frac{1}{2},\chi_D\right)=\frac{P(1)}{2\zeta_{\mathbb{A}}(2)}|D|\left[\log_q|D|+\frac{4}{\log q}\frac{P'}{P}(1)+2\zeta_{\mathbb{A}}\left(\frac{1}{2}\right)\right]+O(|D|^{\frac{3}{4}+\frac{\log_q2}{2}}).
    \end{equation}
\end{thm}
\par\noindent
In this paper, we will use Florea's method to improve the asymptotic formula obtained by Jung (\ref{eq:1.9}). In particular we will obtain secondary main terms of size $gq^{\frac{2g+2}{3}}, q^{\frac{g}{6}+\left[\frac{g}{2}\right]}$ and $q^{\frac{g}{6}+\left[\frac{g-1}{2}\right]}$,  whilst also bounding the error term by $q^{\frac{g}{2}(1+\epsilon)}$. The main result of this paper is the following Theorem. 
\begin{thm}
    Let $q$ be a prime with $q\equiv 1(\text{mod }4)$. Then
   \begin{align}
       \sum_{D\in\mathcal{H}_{2g+2}}L\left(\frac{1}{2},\chi_D\right)&=\frac{P(1)}{2\zeta_{\mathbb{A}}(2)}q^{2g+2}\left[(2g+2)+\frac{4}{\log q}\frac{P'}{P}(1)+2\zeta_{\mathbb{A}}\left(\frac{1}{2}\right)\right]\nonumber\\
       &+q^{\frac{2g+2}{3}}\mathcal{R}(2g+2)+C_1q^{\frac{g}{6}+\left[\frac{g}{2}\right]}+C_2q^{\frac{g}{6}+\left[\frac{g-1}{2}\right]}+O(q^{\frac{g}{2}(1+\epsilon)}),
   \end{align}
   where $\mathcal{R}$ is a polynomial of degree 1 and $C_1$ and $C_2$ are constants that can explicitly be computed (see formulas (\ref{eq:5.35}), (\ref{eq:5.36}) and (\ref{eq:5.37})). 
\end{thm}
\par\noindent
The calculations in this paper will follow the techniques presented in Florea, \cite{Florea2017}. However we encounter extra difficulties and extra terms are present when computing the first moment of $L\left(\frac{1}{2},\chi_D\right)$ for real quadratic function fields, compared to the calculations of the first moment of $L\left(\frac{1}{2},\chi_D\right)$ for imaginary function fields.
\section{Preliminaries and Background}
We first introduce the notation which will be used throughout the article and then provide some background information on Dirichlet L-functions in function fields. We denote $\mathbb{A}^+$ to be the set of all monic polynomials in $\mathbb{F}_q[x]$ and we denote $\mathbb{A}^+_n$ and $\mathbb{A}^+_{\leq n}$ to be the set of all monic polynomials of degree $n$ and degree at most $n$ in $\mathbb{F}_q[x]$ respectively. Let $\mathcal{H}_n$ denote the space of monic, square-free monic polynomials over $\mathbb{F}_q[x]$ of degree $n$. For a polynomial $f\in\mathbb{F}_q[x]$, we denote its degree by $d(f)$ and  its norm by $|f|=q^{d(f)}$. The letter $P$ denotes a monic, irreducible polynomial over $\mathbb{F}_q[x]$.
\subsection{Preliminaries on Dirichlet characters and Dirichlet L-functions for function fields}
Most of the facts in this section are proved in \cite{Rosen2002}. For $\Re(s)>1$, the zeta function of $\mathbb{A}=\mathbb{F}_q[x]$, denoted by $\zeta_{\mathbb{A}}(s)$ is defined by the infinite series
\begin{equation}
    \zeta_{\mathbb{A}}(s):=\sum_{f\in\mathbb{A}^+}\frac{1}{|f|^s}=\prod_P(1-|P|^{-s})^{-1}.
\end{equation}
There are $q^n$ monic polynomials of degree $n$, therefore we have
\begin{equation*}
    \zeta_{\mathbb{A}}(s)=(1-q^{1-s})^{-1}.
\end{equation*}
We will make use of the change of variables $u=q^{-s}$, so that we write $\mathcal{Z}(u)=\zeta_{\mathbb{A}}(s)$ and thus $\mathcal{Z}(u)=(1-qu)^{-1}$. \\
\par\noindent
Assume that $q$ is odd with $q\equiv 1(\text{mod }4)$. For $P$ a monic irreducible polynomial, the quadratic residue symbol $\left(\frac{f}{P}\right)\in\{\pm 1\}$ is defined by 
\begin{equation*}
    \left(\frac{f}{P}\right)\equiv f^{\frac{|P|-1}{2}}\hspace{1cm} (\text{mod }P) 
\end{equation*}
for $f$ coprime to $P$. If $P|f$, then $\left(\frac{f}{P}\right)=0$. We can also define the Jacobi symbol for arbitrary monic $Q$. Let $f$ be coprime to $Q$ and $Q=P_1^{e_1}\dotsc P_s^{e_s}$, then the Jacobi symbol is defined by 
\begin{equation*}
    \left(\frac{f}{Q}\right)=\prod_{j=1}^r\left(\frac{f}{P_j}\right)^{e_j}.
\end{equation*}
\begin{thm}[Quadratic Reciprocity] 
Let $A,B\in\mathbb{F}_q[x]$ be relatively prime and $A\neq 0$ and $B\neq 0$. Then
\begin{equation*}
    \left(\frac{A}{B}\right)=\left(\frac{B}{A}\right)(-1)^{\frac{q-1}{2}d(A)d(B)}.
\end{equation*}
\end{thm}
\par\noindent
If we assume that $q\equiv 1(\text{mod }4)$, then the quadratic reciprocity gives
\begin{equation*}
    \left(\frac{A}{B}\right)=\left(\frac{B}{A}\right).
\end{equation*}
For $g\geq 1$ we have
\begin{equation*}
    |\mathcal{H}_{2g+2}|=(q-1)q^{2g+1}=\frac{|D|}{\zeta_{\mathbb{A}}(2)}.
\end{equation*}
\begin{defn}
Let $D\in\mathbb{F}_q[x]$ be square-free. We define the quadratic character using the quadratic residue symbol for $\mathbb{F}_q[x]$ by 
\begin{equation*}
   \chi_D(f)=\left(\frac{D}{f}\right). 
\end{equation*}
Therefore, if $P\in\mathbb{F}_q[x]$, we have 
\begin{equation*}
    \chi_D(P)=\begin{cases}
    0,& \text{if }P|D,\\
    1,&\text{if }P\nmid D\text{ and }D\text{ is a square modulo }P,\\
    -1,& \text{ if }P\nmid D\text{ and }D\text{ is a non-square modulo }P.
    \end{cases}
\end{equation*}
\end{defn}
\begin{defn}
The L-function corresponding to the quadratic character $\chi_D$ by 
\begin{equation*}
    L(s,\chi_D)=\sum_{f\in\mathbb{A}^+}\frac{\chi_D(f)}{|f|^s},
\end{equation*}
which converges for $\Re(s)>1$. For the change of variables $u=q^{-s}$, we have
\begin{equation}
    L(s,\chi_D)=\mathcal{L}(u,\chi_D)=\sum_{f\in\mathbb{A}^+}\chi_D(f)u^{d(f)}=\prod_P(1-\chi_D(f)u^{d(P)})^{-1}. 
\end{equation}
\end{defn}
\par\noindent
Since $D$ is a monic, square-free polynomial, we have, from Proposition 4.3 in \cite{Rosen2002}, that $\mathcal{L}(u,\chi_D)$ is a polynomial in $u$ of $d(D)-1$. From \cite{Rudnick2008TracesEnsemble}, $\mathcal{L}(u,\chi_D)$ has a trivial zero if and only if $d(D)$ is even. This enables us to define the completed L-function, $\mathcal{L}^*(u,\chi_D)$, by 
\begin{equation}
    \mathcal{L}(u,\chi_D)=(1-u)^{\lambda}\mathcal{L}^*(u,\chi_D),
\end{equation}
where $\lambda=1$ if $d(D)$ is even and $\lambda=0$ otherwise. Then $\mathcal{L}^*(u,\chi_D)$ is a polynomial in $u$ of degree $2\delta=d(D)-1-\lambda$ and satisfies the functional equation 
\begin{equation}
    \mathcal{L}^*(u,\chi_D)=(qu^2)^{\delta}\mathcal{L}^*((qu)^{-1},\chi_D).
\end{equation}
\par\noindent
For $D$ a monic, square-free polynomial of degree $2g+1$ or $2g+2$, the affine equation $y^2=D(x)$ defines a projective and connected hyperelliptic curve $C_D$ of genus $g$ over $\mathbb{F}_q$. The zeta function associated to $C_D$ was first introduced by Artin, \cite{Artin1924}, and is defined by
\begin{equation}
    Z_{C_D}(u)=\exp\left(\sum_{n=1}^{\infty}N_n(C_D)\frac{u^n}{n}\right),
\end{equation}
where $N_n(C_D)$ is the number of points on $C_D$ with coordinates in a field extension $\mathbb{F}_{q^n}$ of $\mathbb{F}_q$ of degree $n\geq 1$. Weil, \cite{Weil1948}, showed that 
\begin{equation*}
    Z_{C_D}(u)= \frac{P_{C_D}(u)}{(1-u)(1-qu)},
\end{equation*}
where $P_{C_D}(u)$ is a polynomial of degree $2g$. In his thesis, Artin, proved that $P_{C_D}(u)=\mathcal{L}^*(u,\chi_D)$. Also, Weil, \cite{Weil1948}, proved the Riemann Hypothesis for function fields, thus all the zeros of $\mathcal{L}^*(u,\chi_D)$ lie on the circle $|u|=q^{-\frac{1}{2}}$.
\subsection{Functional Equation and Preliminary Lemmas}
For $D\in\mathcal{H}_{2g+2}$, the approximate functional equation was initially proved in Jung, \cite{Jung2013NoteEnsemble}, but has been corrected to match that of \cite{Rubinstein2015}. 
\begin{lemma}
Let $\chi_D$ be a quadratic character, where $D\in\mathcal{H}_{2g+2}$. Then
\begin{align*}
   L\left(\frac{1}{2},\chi_D\right)&=\sum_{n=0}^g\sum_{f\in\mathbb{A}^+_n}\chi_D(f)q^{-\frac{n}{2}}-q^{-\frac{g+1}{2}}\sum_{n=0}^g\sum_{f\in\mathbb{A}^+_n}\chi_D(f)\\
    &+\sum_{n=0}^{g-1}\sum_{f\in\mathbb{A}^+_n}\chi_D(f)q^{-\frac{n}{2}}-q^{-\frac{g}{2}}\sum_{n=0}^{g-1}\sum_{f\in\mathbb{A}^+_n}\chi_D(f).
\end{align*}
\end{lemma}
\begin{proof}
See \cite{Jung2013NoteEnsemble}, Lemma 2.1.
\end{proof}
\par\noindent
Using Lemma 2.4, it follows that
\begin{align}\label{eq:2.6}
    \sum_{D\in\mathcal{H}_{2g+2}}L\left(\frac{1}{2},\chi_D\right)&=\sum_{f\in\mathbb{A}^+_{\leq g}}\frac{1}{\sqrt{|f|}}\sum_{D\in\mathcal{H}_{2g+2}}\chi_D(f)-q^{-\frac{g+1}{2}}\sum_{f\in\mathbb{A}^+_{\leq g}}\sum_{D\in\mathcal{H}_{2g+2}}\chi_D(f)\nonumber\\
    &+\sum_{f\in\mathbb{A}^+_{\leq g-1}}\frac{1}{\sqrt{|f|}}\sum_{D\in\mathcal{H}_{2g+2}}\chi_D(f)-q^{-\frac{g}{2}}\sum_{f\in\mathbb{A}^+_{\leq g-1}}\sum_{D\in\mathcal{H}_{2g+2}}\chi_D(f).
\end{align}
We now state two Lemmas that will be used in the calculations later. 
\begin{lemma}
Let $f\in\mathbb{A}^+$. Then we have
\begin{equation}
    \sum_{D\in\mathcal{H}_{2g+2}}\chi_D(f)=\sum_{\substack{C|f^{\infty}\\C\in\mathbb{A}^+_{\leq g+1}}}\sum_{h\in\mathbb{A}^+_{2g+2-2d(C)}}\chi_f(h)-q\sum_{\substack{C|f^{\infty}\\C\in\mathbb{A}^+_{\leq g}}}\sum_{h\in\mathbb{A}^+_{2g-2d(C)}}\chi_f(h),
\end{equation}
where $C|f^{\infty}$ means that any prime factor of $C$ are among the prime factors of $f$.
\end{lemma}
\begin{proof}
See \cite{Florea2017}, Lemma 2.2
\end{proof}
\par\noindent
We now state a version of Poisson summation formula over function fields. For $a\in\mathbb{F}_q((\frac{1}{t}))$, let $e(a):=e^{2\pi ia_1/q}$, where $a_1$ is the coefficient of $\frac{1}{t}$ in the expansion of $a$ (for more information, see \cite{Florea2017}, section 3). For $\chi$ a general character modulo $f$, the generalised Gauss sum $G(V,\chi_f)$ is defined as 
\begin{equation}
    G(V,\chi)=\sum_{A\text{ mod } f}\chi(A)e\left(\frac{AV}{f}\right).
\end{equation}
The following Poisson summation formula holds.
\begin{lemma}
Let $f\in\mathbb{A}^+$ and let $m$ be a positive integer. 
\begin{enumerate}
    \item If $d(f)$ is odd, then 
    \begin{equation}
        \sum_{g\in\mathbb{A}^+_m}\chi_f(g)=\frac{q^{m+\frac{1}{2}}}{|f|}\sum_{V\in\mathbb{A}^+_{d(f)-m-1}}G(V,\chi_f).
    \end{equation}
    \item If $d(f)$ is even, then 
    \begin{equation}
        \sum_{g\in\mathbb{A}^+_m}\chi_f(g)=\frac{q^m}{|f|}\left(G(0,\chi_f)+(q-1)\sum_{V\in\mathbb{A}^+_{\leq d(f)-m-2}}G(V,\chi_f)-\sum_{V\in\mathbb{A}^+_{d(f)-m-1}}G(V,\chi_f)\right).
    \end{equation}
\end{enumerate}
\end{lemma}
\begin{proof}
See \cite{Florea2017}, Proposition 3.1.
\end{proof}
\begin{remark}
$G(0,\chi_f)$ is nonzero if and only if $f$ is a square, in which case $G(0,\chi_f)=\phi(f)$, where $\phi(f)$ is Euler's phi function for polynomials in $\mathbb{F}_q[x]$.
\end{remark}
\par\noindent
\begin{lemma}[The function field analogue of Perron's formula]
If the power series 
\begin{equation*}
H(u)=\sum_{f\in\mathbb{A}^+}a(f)u^{d(f)}
\end{equation*}
converges absolutely for $|u|\leq R<1$, then 
\begin{equation}
    \sum_{f\in\mathbb{A}^+_n}a(f)=\frac{1}{2\pi i}\oint_{|u|=R}\frac{H(u)}{u^{n+1}}du
\end{equation}
and
\begin{equation}
    \sum_{f\in\mathbb{A}^+_{\leq n}}a(f)=\frac{1}{2\pi i}\oint_{|u|=R}\frac{H(u)}{(1-u)u^{n+1}}du.
\end{equation}
\end{lemma}
\section{Setup of the Problem}
Using Lemma 2.5 and the approximate functional equation (\ref{eq:2.6}), we write
\begin{equation}\label{firsteq}
    \sum_{D\in\mathcal{H}_{2g+2}}L\left(\frac{1}{2},\chi_D\right)=\mathcal{S}_{g,1}-\mathcal{S}_{g,2}+\mathcal{S}_{g-1,1}-\mathcal{S}_{g-1,2},
\end{equation}
where
\begin{equation*}
    \mathcal{S}_{g,1}=\sum_{f\in\mathbb{A}^+_{\leq g}}\frac{1}{\sqrt{|f|}}\sum_{\substack{C|f^{\infty}\\C\in\mathbb{A}^+_{\leq g+1}}}\sum_{h\in\mathbb{A}^+_{2g+2-2d(C)}}\chi_f(h)-q\sum_{f\in\mathbb{A}^+_{\leq g}}\frac{1}{\sqrt{|f|}}\sum_{\substack{C|f^{\infty}\\C\in\mathbb{A}^+_{\leq g}}}\sum_{h\in\mathbb{A}^+_{2g-2d(C)}}\chi_f(h),
\end{equation*}
\begin{equation*}
    \mathcal{S}_{g,2}=q^{-\frac{g+1}{2}}\sum_{f\in\mathbb{A}^+_{\leq g}}\sum_{\substack{C|f^{\infty}\\C\in\mathbb{A}^+_{\leq g+1}}}\sum_{h\in\mathbb{A}^+_{2g+2-2d(C)}}\chi_f(h)-q\sum_{f\in\mathbb{A}^+_{\leq g}}\sum_{\substack{C|f^{\infty}\\C\in\mathbb{A}^+_{\leq g}}}\sum_{h\in\mathbb{A}^+_{2g-2d(C)}}\chi_f(h),
\end{equation*}
\begin{equation*}
    \mathcal{S}_{g-1,1}=\sum_{f\in\mathbb{A}^+_{\leq g-1}}\frac{1}{\sqrt{|f|}}\sum_{\substack{C|f^{\infty}\\C\in\mathbb{A}^+_{\leq g+1}}}\sum_{h\in\mathbb{A}^+_{2g+2-2d(C)}}\chi_f(h)-q\sum_{f\in\mathbb{A}^+_{\leq g-1}}\frac{1}{\sqrt{|f|}}\sum_{\substack{C|f^{\infty}\\C\in\mathbb{A}^+_{\leq g}}}\sum_{h\in\mathbb{A}^+_{2g-2d(C)}}\chi_f(h)
\end{equation*}
and
\begin{equation*}
    \mathcal{S}_{g-1,2}=q^{-\frac{g}{2}}\sum_{f\in\mathbb{A}^+_{\leq g-1}}\sum_{\substack{C|f^{\infty}\\C\in\mathbb{A}^+_{\leq g+1}}}\sum_{h\in\mathbb{A}^+_{2g+2-2d(C)}}\chi_f(h)-q\sum_{f\in\mathbb{A}^+_{\leq g-1}}\sum_{\substack{C|f^{\infty}\\C\in\mathbb{A}^+_{\leq g}}}\sum_{h\in\mathbb{A}^+_{2g-2d(C)}}\chi_f(h).
\end{equation*}
From section 4 in \cite{Florea2017}, we have 
\begin{equation*}
    \sum_{\substack{C|f^{\infty}\\C\in\mathbb{A}^+_{g+1}}}1\ll q^{g\epsilon}
\end{equation*}
we see that the terms in $\mathcal{S}_{g,1}, \mathcal{S}_{g,2}, \mathcal{S}_{g-1,1}$ and $\mathcal{S}_{g-1,2}$ corresponding to $C\in\mathbb{A}^+_{g+1}$ are bounded by $O(q^{\frac{g}{2}(1+\epsilon)})$. Therefore we can rewrite the terms as 
\begin{equation}
    \mathcal{S}_{g,1}=\sum_{f\in\mathbb{A}^+_{\leq g}}\frac{1}{\sqrt{|f|}}\sum_{\substack{C|f^{\infty}\\C\in\mathbb{A}^+_{\leq g}}}\left(\sum_{h\in\mathbb{A}^+_{2g+2-2d(C)}}\chi_f(h)-q\sum_{h\in\mathbb{A}^+_{2g-2d(C)}}\chi_f(h)\right)+O(q^{\frac{g}{2}(1+\epsilon)}),
\end{equation}
\begin{equation}
    \mathcal{S}_{g,2}=q^{-\frac{g+1}{2}}\sum_{f\in\mathbb{A}^+_{\leq g}}\sum_{\substack{C|f^{\infty}\\C\in\mathbb{A}^+_{\leq g}}}\left(\sum_{h\in\mathbb{A}^+_{2g+2-2d(C)}}\chi_f(h)-q\sum_{h\in\mathbb{A}^+_{2g-2d(C)}}\chi_f(h)\right)+O(q^{\frac{g}{2}(1+\epsilon)}),
\end{equation}
\begin{equation}
    \mathcal{S}_{g-1,1}=\sum_{f\in\mathbb{A}^+_{\leq g-1}}\frac{1}{\sqrt{|f|}}\sum_{\substack{C|f^{\infty}\\C\in\mathbb{A}^+_{\leq g}}}\left(\sum_{h\in\mathbb{A}^+_{2g+2-2d(C)}}\chi_f(h)-q\sum_{h\in\mathbb{A}^+_{2g-2d(C)}}\chi_f(h)\right)+O(q^{\frac{g}{2}(1+\epsilon)})
\end{equation}
and
\begin{equation}
    \mathcal{S}_{g-1,2}=q^{-\frac{g}{2}}\sum_{f\in\mathbb{A}^+_{\leq g}}\sum_{\substack{C|f^{\infty}\\C\in\mathbb{A}^+_{\leq g}}}\left(\sum_{h\in\mathbb{A}^+_{2g+2-2d(C)}}\chi_f(h)-q\sum_{h\in\mathbb{A}^+_{2g-2d(C)}}\chi_f(h)\right)+O(q^{\frac{g}{2}(1+\epsilon)}).
\end{equation}
For $k\in\{g,g-1\}$ and $\ell\in\{1,2\}$, write
\begin{equation}\label{secondeq}
    \mathcal{S}_{k,\ell}=\mathcal{S}^o_{k,\ell}+\mathcal{S}^e_{k,\ell}+O(q^{\frac{g}{2}(1+\epsilon)}),
\end{equation} where $\mathcal{S}^o_{k,\ell}$ and $\mathcal{S}^e_{k,\ell}$ denotes the sum over $f\in\mathbb{A}^+_{\leq k}$ of odd and even degree respectively. If $d(f)$ is odd, then using Lemma 2.6, we have
\begin{equation}
    \mathcal{S}^o_{g,1}=q^{2g+\frac{5}{2}}\sum_{\substack{f\in\mathbb{A}^+_{\leq g}\\d(f)\text{ odd}}}\frac{1}{|f|}\sum_{\substack{C|f^{\infty}\\C\in\mathbb{A}^+_{\leq g}}}\frac{1}{|C|^2}S^o(V;f,C),
\end{equation}
\begin{equation}
    \mathcal{S}^o_{g,2}=q^{\frac{3g}{2}+2}\sum_{\substack{f\in\mathbb{A}^+_{\leq g}\\d(f)\text{ odd}}}\frac{1}{\sqrt{|f|}}\sum_{\substack{C|f^{\infty}\\C\in\mathbb{A}^+_{\leq g}}}\frac{1}{|C|^2}S^o(V;f,C),
\end{equation}
\begin{equation}
    \mathcal{S}^o_{g-1,1}=q^{2g+\frac{5}{2}}\sum_{\substack{f\in\mathbb{A}^+_{\leq g-1}\\d(f)\text{ odd}}}\frac{1}{|f|}\sum_{\substack{C|f^{\infty}\\C\in\mathbb{A}^+_{\leq g}}}\frac{1}{|C|^2}S^o(V;f,C)
\end{equation}
and
\begin{equation}
   \mathcal{S}^o_{g-1,2}=q^{\frac{3g+5}{2}}\sum_{\substack{f\in\mathbb{A}^+_{\leq g-1}\\d(f)\text{ odd}}}\frac{1}{\sqrt{|f|}}\sum_{\substack{C|f^{\infty}\\C\in\mathbb{A}^+_{\leq g}}}\frac{1}{|C|^2}S^o(V;f,C),
\end{equation}
where 
\begin{equation}\label{eq:3.9}
    S^o(V;f,C)=\sum_{V\in\mathbb{A}^+_{d(f)-2g-3+2d(C)}}\frac{G(V,\chi_f)}{\sqrt{|f|}}-\frac{1}{q}\sum_{V\in\mathbb{A}^+_{d(f)-2g-1+2d(C)}}\frac{G(V,\chi_f)}{\sqrt{|f|}}.
\end{equation}
\par\noindent
If $d(f)$ is even, then we rewrite $\mathcal{S}^e_{k,\ell}$ as
\begin{equation}\label{thirdeq}
    \mathcal{S}^e_{k,\ell}=M_{k,\ell}+S^e_{k,\ell,1}+S^e_{k,\ell,2}.
\end{equation}
Using the remark from the previous section, we have
\begin{equation}\label{eq:3.10}
    M_{g,1}=\frac{q^{2g+2}}{\zeta_{\mathbb{A}}(2)}\sum_{L\in\mathbb{A}^+_{\leq \left[\frac{g}{2}\right]}}\frac{\phi(L^2)}{|L|^3}\sum_{\substack{C|f^{\infty}\\C\in\mathbb{A}^+_{\leq g}}}\frac{1}{|C|^2},
    \end{equation}
    \begin{equation}\label{eq:3.11}
    M_{g,2}=\frac{q^{\frac{3g+3}{2}}}{\zeta_{\mathbb{A}}(2)}\sum_{L\in\mathbb{A}^+_{\leq \left[\frac{g}{2}\right]}}\frac{\phi(L^2)}{|L|^2}\sum_{\substack{C|f^{\infty}\\C\in\mathbb{A}^+_{\leq g}}}\frac{1}{|C|^2},
\end{equation}
\begin{equation}\label{eq:3.12}
    M_{g-1,1}=\frac{q^{2g+2}}{\zeta_{\mathbb{A}}(2)}\sum_{L\in\mathbb{A}^+_{\leq \left[\frac{g-1}{2}\right]}}\frac{\phi(L^2)}{|L|^3}\sum_{\substack{C|f^{\infty}\\C\in\mathbb{A}^+_{\leq g}}}\frac{1}{|C|^2}
\end{equation}
and
\begin{equation}\label{eq:3.13}
    M_{g-1,2}=\frac{q^{\frac{3g}{2}+2}}{\zeta_{\mathbb{A}}(2)}\sum_{L\in\mathbb{A}^+_{\leq \left[\frac{g-1}{2}\right]}}\frac{\phi(L^2)}{|L|^2}\sum_{\substack{C|f^{\infty}\\C\in\mathbb{A}^+_{\leq g}}}\frac{1}{|C|^2}.
\end{equation}
Similarly, for $j\in\{1,2\}$ we have
\begin{equation}
    S^e_{g,1,j}=q^{2g+2}\sum_{\substack{f\in\mathbb{A}^+_{\leq g}\\d(f)\text{ even}}}\frac{1}{|f|}\sum_{\substack{C|f^{\infty}\\C\in\mathbb{A}^+_{\leq g}}}\frac{1}{|C|^2}S^e_j(V;f,C),
\end{equation}
\begin{equation}
    S^e_{g,2,j}=q^{\frac{3g+3}{2}}\sum_{\substack{f\in\mathbb{A}^+_{\leq g}\\d(f)\text{ even}}}\frac{1}{\sqrt{|f|}}\sum_{\substack{C|f^{\infty}\\C\in\mathbb{A}^+_{\leq g}}}\frac{1}{|C|^2}S^e_j(V;f,C),
\end{equation}
\begin{equation}
    S^e_{g-1,1,j}=q^{2g+2}\sum_{\substack{f\in\mathbb{A}^+_{\leq g-1}\\d(f)\text{ even}}}\frac{1}{|f|}\sum_{\substack{C|f^{\infty}\\C\in\mathbb{A}^+_{\leq g}}}\frac{1}{|C|^2}S^e_j(V;f,C)
\end{equation}
and
\begin{equation}
    S^e_{g-1,2,j}=q^{\frac{3g}{2}+2}\sum_{\substack{f\in\mathbb{A}^+_{\leq g-1}\\d(f)\text{ even}}}\frac{1}{\sqrt{|f|}}\sum_{\substack{C|f^{\infty}\\C\in\mathbb{A}^+_{\leq g}}}\frac{1}{|C|^2}S^e_j(V;f,C),
\end{equation}
where 
\begin{equation}\label{eq:3.18}
    S_1^e(V;f,C)=(q-1)\sum_{V\in\mathbb{A}^+_{\leq d(f)-2g-4+2d(C)}}\frac{G(V,\chi_f)}{\sqrt{|f|}}-\frac{q-1}{q}\sum_{V\in\mathbb{A}^+_{\leq d(f)-2g-2+2d(C)}}\frac{G(V,\chi_f)}{\sqrt{|f|}}
\end{equation}
and 
\begin{equation}\label{eq:3.22}
  S^e_2(V;f,C)=\frac{1}{q}\sum_{V\in\mathbb{A}^+_{d(f)-2g-1+2d(C)}}\frac{G(V,\chi_f)}{\sqrt{|f|}}-\sum_{V\in\mathbb{A}^+_{d(f)-2g-3+2d(C)}}\frac{G(V,\chi_f)}{\sqrt{|f|}}.
\end{equation}
\par\noindent
Define $S^i_{k,\ell}(V=\square)$ to be the sum over $V$ square and $S^i_{k,\ell}(V\neq\square)$ to be the sum over non-square $V$. Note that in Equation (\ref{eq:3.22}), when $d(f)$ is even, $d(V)$ is odd and so $V$ cannot be a square. Also note that in Equation (\ref{eq:3.9}), when $d(f)$ is odd, $d(V)$ is even, thus there is a contribution to the main term when $d(f)$ is odd, which is not present when working in the imaginary function field case. 
\section{Main Term}
In this section we evaluate the main terms, $M_{g,1}, M_{g,2}, M_{g-1,1}$ and $M_{g-1,2}$. The main result in this section is the following result.
\begin{prop}
For any $\epsilon>0$, we have that
\begin{equation*}
     M_{g,1}=\frac{q^{2g+2}}{\zeta_{\mathbb{A}}(2)}\frac{1}{2\pi i}\oint_{|u|=r}\frac{\mathcal{C}(u)}{u(1-qu)^2(qu)^{\left[\frac{g}{2}\right]}}du+O(q^{g\epsilon}),
\end{equation*}
   \begin{equation*}
      M_{g-1,1}=\frac{q^{2g+2}}{\zeta_{\mathbb{A}}(2)}\frac{1}{2\pi i}\oint_{|u|=r}\frac{\mathcal{C}(u)}{u(1-qu)^2(qu)^{\left[\frac{g-1}{2}\right]}}du+O(q^{g\epsilon}), 
   \end{equation*}
   \begin{equation*}
       M_{g,2}=\frac{q^{\frac{3g+3}{2}}}{\zeta_{\mathbb{A}}(2)}\frac{1}{2\pi i}\oint_{|u|=r}\frac{\mathcal{C}(u)}{u(1-u)(1-qu)u^{\left[\frac{g}{2}\right]}}du+O(q^{g\epsilon})
   \end{equation*}
and
\begin{equation*}
     M_{g-1,2}=\frac{q^{\frac{3g}{2}+2}}{\zeta_{\mathbb{A}}(2)}\frac{1}{2\pi i}\oint_{|u|=r}\frac{\mathcal{C}(u)}{u(1-u)(1-qu)u^{\left[\frac{g-1}{2}\right]}}du+O(q^{g\epsilon}),
\end{equation*}
where $r<q^{-1}$ and 
\begin{equation}
    \mathcal{C}(u)=\prod_P\left(1-\frac{u^{d(P)}}{|P|+1}\right).
\end{equation}
\end{prop}
\begin{remark}
Let 
\begin{equation}\label{eq:4.1}
    M=M_{g,1}-M_{g,2}+M_{g-1,1}-M_{g-1,2}.
\end{equation}
\end{remark}
\begin{remark}
$\mathcal{C}(u)$ is analytic in $|u|<1$. We may further write 
\begin{equation}\label{eq:4.4}
    \mathcal{C}(u)=\mathcal{Z}\left(\frac{u}{q}\right)^{-1}\prod_P\left(1+\frac{u^{d(P)}}{(1+|P|)(|P|-u^{d(P)})}\right)=(1-u)\prod_P\left(1+\frac{u^{d(P)}}{(1+|P|)(|P|-u^{d(P)})}\right)
\end{equation}
which furnishes an analytic continuation of $\mathcal{C}(u)$ to the region $|u|<q$. 
\end{remark}
\begin{proof}[Proof of Proposition 4.1]
From (\ref{eq:3.10}), (\ref{eq:3.11}), (\ref{eq:3.12}) and (\ref{eq:3.13}) and using the facts that (see \cite{Florea2017}, Section 5),
\begin{equation}
    \sum_{\substack{C|f^{\infty}\\C\in\mathbb{A}^+_{\leq g}}}\frac{1}{|C|^2}=\prod_{P|f}(1-|P|^{-2})^{-1}+O(q^{-g(2-\epsilon)})\text{ and }\frac{\phi(L^2)}{|L|^2}=\prod_{P|L}(1-|P|^{-1}),
\end{equation}
we have
\begin{equation}
    M_{g,1}=\frac{q^{2g+2}}{\zeta_{\mathbb{A}}(2)}\sum_{L\in\mathbb{A}^+_{\leq\left[\frac{g}{2}\right]}}\frac{1}{|L|}\prod_{P|L}\frac{|P|}{|P|+1}+O(q^{g\epsilon}),
\end{equation}
\begin{equation}
    M_{g,2}=\frac{q^{\frac{3g+3}{2}}}{\zeta_{\mathbb{A}}(2)}\sum_{L\in\mathbb{A}^+_{\leq \left[\frac{g}{2}\right]}}\prod_{P|L}\frac{|P|}{|P|+1}+O(q^{g\epsilon}),
\end{equation}
\begin{equation}
    M_{g-1,1}=\frac{q^{2g+2}}{\zeta_{\mathbb{A}}(2)}\sum_{L\in\mathbb{A}^+_{\leq \left[\frac{g-1}{2}\right]}}\frac{1}{|L|}\prod_{P|L}\frac{|P|}{|P|+1}+O(q^{g\epsilon})
\end{equation}
and
\begin{equation}
    M_{g-1,2}=\frac{q^{\frac{3g}{2}+2}}{\zeta_{\mathbb{A}}(2)}\sum_{L\in\mathbb{A}^+_{\leq \left[\frac{g-1}{2}\right]}}\prod_{P|L}\frac{|P|}{|P|+1}+O(q^{g\epsilon}).
\end{equation}
Using the function field analogue of Perron's formula (Lemma 2.8), we have 
\begin{equation}\label{eq:4.10}
    M_{g,1}=\frac{q^{2g+2}}{\zeta_{\mathbb{A}}(2)}\frac{1}{2\pi i}\oint_{|u|=r}\frac{\mathcal{A}(u)}{u(1-qu)(qu)^{\left[\frac{g}{2}\right]}}du,
\end{equation}
\begin{equation}\label{eq:4.11}
  M_{g,2}=\frac{q^{\frac{3g+3}{2}}}{\zeta_{\mathbb{A}}(2)}\frac{1}{2\pi i}\oint_{|u|=r}\frac{\mathcal{A}(u)}{u(1-u)u^{\left[\frac{g}{2}\right]}}du,
\end{equation}
\begin{equation}\label{eq:4.12}
    M_{g-1,1}=\frac{q^{2g+2}}{\zeta_{\mathbb{A}}(2)}\frac{1}{2\pi i}\oint_{|u|=r}\frac{\mathcal{A}(u)}{u(1-qu)(qu)^{\left[\frac{g-1}{2}\right]}}du
\end{equation}
and
\begin{equation}\label{eq:4.13}
    M_{g-1,2}=\frac{q^{\frac{3g}{2}+2}}{\zeta_{\mathbb{A}}(2)}\frac{1}{2\pi i}\oint_{|u|=r}\frac{\mathcal{A}(u)}{u(1-u)u^{\left[\frac{g-1}{2}\right]}}du,
\end{equation}
where $r<q^{-1}$ and
\begin{equation}
    \mathcal{A}(u)=\sum_{L\in\mathbb{A}^+}u^{d(L)}\prod_{P}\frac{|P|}{|P|+1}.
\end{equation}
By multiplicativity, we may write
\begin{equation}\label{eq:4.15}
    \mathcal{A}(u)=\prod_P\left(1+\frac{|P|}{|P|+1}\frac{u^{d(P)}}{1-u^{d(P)}}\right)=\mathcal{Z}(u)\mathcal{C}(u)=\frac{\mathcal{C}(u)}{(1-qu)}.
\end{equation}
Inserting (\ref{eq:4.15}) into (\ref{eq:4.10}), (\ref{eq:4.11}), (\ref{eq:4.12}) and (\ref{eq:4.13}) the Proposition follows.
\end{proof}
\section{Contribution from V-square}
Let
\begin{equation}\label{eq:5.1}
        \mathcal{S}(V=\square)=\mathcal{S}^o(V=\square)+\mathcal{S}^e(V=\square)
\end{equation}
where
\begin{equation}\label{eq:5.2}
    \mathcal{S}^o(V=\square)=\mathcal{S}^o_{g,1}(V=\square)-\mathcal{S}^o_{g,2}(V=\square)+\mathcal{S}^o_{g-1,1}(V=\square)-\mathcal{S}^o_{g-1,2}(V=\square)
\end{equation}
and
\begin{equation}\label{eq:5.3}
    S^e(V=\square)=\mathcal{S}^e_{g,1}(V=\square)-\mathcal{S}^e_{g,2}(V=\square)+\mathcal{S}^e_{g-1,1}(V=\square)-\mathcal{S}^e_{g-1,2}(V=\square).
\end{equation}
In this section we will evaluate the term $\mathcal{S}(V=\square)$. The next Proposition is the main result in this section.
\begin{prop}
Using the same notation as before, we have that 
\begin{align}
    \mathcal{S}(V=\square)&=\mathcal{S}_1(V=\square)+\mathcal{S}_2(V=\square)+\mathcal{S}_3(V=\square)+\mathcal{S}_4(V=\square)\nonumber\\&+q^{\frac{2g+2}{3}}\mathcal{R}(2g+2)+q^{\frac{g}{6}+\left[\frac{g}{2}\right]}C_1+q^{\frac{g}{6}+\left[\frac{g-1}{2}\right]}C_2+O(q^{\frac{g}{2}(1+\epsilon)}),
    \end{align}
    where 
    \begin{equation*}
    \mathcal{S}_1(V=\square)=-\frac{q^{2g+2}}{\zeta_{\mathbb{A}}(2)}\frac{1}{2\pi i}\oint_{|u|=R}\frac{\mathcal{C}(u)}{u(1-qu)^2(qu)^{\left[\frac{g}{2}\right]}}du,
    \end{equation*}
    \begin{equation*}
        \mathcal{S}_2(V=\square)=-\frac{q^{2g+2}}{\zeta_{\mathbb{A}}(2)}\frac{1}{2\pi i}\oint_{|u|=R}\frac{\mathcal{C}(u)}{u(1-qu)^2(qu)^{\left[\frac{g-1}{2}\right]}}du, 
    \end{equation*}
\begin{equation*}
    \mathcal{S}_3(V=\square)=\frac{q^{\frac{3g+3}{2}}}{\zeta_{\mathbb{A}}(2)}\frac{1}{2\pi i}\oint_{|u|=R}\frac{\mathcal{C}(u)}{u(1-u)(1-qu)u^{\left[\frac{g}{2}\right]}}du 
\end{equation*}
  and
  \begin{equation*}
    \mathcal{S}_4(V=\square)=\frac{q^{\frac{3g}{2}+2}}{\zeta_{\mathbb{A}}(2)}\frac{1}{2\pi i}\oint_{|u|=R}\frac{\mathcal{C}(u)}{u(1-u)(1-qu)u^{\left[\frac{g-1}{2}\right]}}du,
\end{equation*}
with $1<R<q$ and
\begin{equation*}
    \mathcal{C}(u)=\prod_P\left(1-\frac{u^{d(P)}}{|P|+1}\right).
\end{equation*}
Furthermore $\mathcal{R}$ is a linear polynomial and $C_1$ and $C_2$ are constants that can be explicitly calculated, (see (\ref{eq:5.35}), (\ref{eq:5.36}) and (\ref{eq:5.37})).
\end{prop}
\par\noindent
Before we prove Proposition 5.1, we need the following notation and subsequent results. For $|z|>q^{-2}$, let
\begin{equation*}
    \mathcal{B}(z,w)=\sum_{f\in\mathbb{A}^+}w^{d(f)}A_f(z)\prod_{P|f}(1-|P|^{-2}z^{-d(P)})^{-1},
\end{equation*}
where
\begin{equation*}
    A_f(z)=\sum_{l\in\mathbb{A}^+}z^{d(l)}\frac{G(l^2,\chi_f)}{\sqrt{|f|}}.
\end{equation*}
Then we have the following results.
\begin{lemma}
For $|z|>q^{-2}$, we have 
\begin{equation}\label{eq:5.5}
    \mathcal{B}(z,w)=\mathcal{Z}(z)\mathcal{Z}(w)\mathcal{Z}(qw^2z)\prod_P\mathcal{B}_P(z,w),
\end{equation}
where
\begin{equation*}
    \mathcal{B}_P(z,w)=1+\frac{w^{d(P)}-(zw^2)^{d(P)}|P|^2-(z^2w)^{d(P)}|P|^2+(z^2w^3)^{d(P)}|P|^2+(zw^2)^{d(P)}|P|-(zw^3)^{d(P)}|P|}{z^{d(P)}|P|^2-1}.
\end{equation*}
Moreover $\prod_P\mathcal{B}_P(z,w)$ converges absolutely for $|w|<q|z|, |w|<q^{-\frac{1}{2}}$ and $|wz|<q^{-1}$.
\end{lemma}
\begin{proof}
See \cite{Florea2017}, Lemma 6.2.
\end{proof}
\begin{lemma}
We have
\begin{equation}\label{eq:5.6}
    \prod_P\mathcal{B}_P(z,w)=\mathcal{Z}\left(\frac{w}{q^2z}\right)\mathcal{Z}(w^2)^{-1}\prod_P\mathcal{D}_P(z,w),
\end{equation}
where
\begin{align*}
    &\mathcal{D}_P(z,w)\\&=1+\frac{-w^{2d(P)}-\frac{w^{3d(P)}}{|P|}+\frac{w^{d(P)}}{z^{d(P)}|P|^2}+(zw^2)^{d(P)}|P|+(zw^2)^{d(P)}-(z^2w)^{d(P)}|P|^2+(zw^3)^{d(P)}-(z^2w^2)^{d(P)}|P|^2}{(z^{d(P)}|P|^2-1)(1+w^{d(P)})}.
\end{align*}
Moreover $\mathcal{D}_P(z,w)$ converges absolutely for $|w|^2<q|z|, |w|<q^3|z|^2, |w|<1$ and $|wz|<q^{-1}$. 
\end{lemma}
\begin{proof}
See \cite{Florea2017}, Lemma 6.3
\end{proof}
\par\noindent
\textbf{Outline of the Proof of Proposition 5.1:} From the Poisson summation formula the sum over square polynomials $V$ will occur when the degree of $f$ is even and when the degree of $f$ is odd. In the next two subsections, we will find two integrals for each $\mathcal{S}^i_{k,\ell}(V=\square)$ corresponding to simple poles at $w=q^{-1}$ and $w=qz$. In the third subsection we will manipulate the integrals corresponding to the pole at $w=q^{-1}$, similar to that done in section 6, \cite{Florea2017}, which will yield the main terms. In the final subsection, we will evaluate the integrals corresponding to the pole at $w=qz$, which will yield the secondary main terms. 
\subsection{Degree $f$ even}
In this subsection, we prove the following result.
\begin{lemma}
We have
\begin{align}
    \mathcal{S}^e(V=\square)=\mathcal{A}^e_{g,1}-\mathcal{A}^e_{g,2}+\mathcal{A}^e_{g-1,1}-\mathcal{A}^e_{g-1,2}+\mathcal{B}^e_{g,1}-\mathcal{B}^e_{g,2}+\mathcal{B}^e_{g-1,1}-\mathcal{B}^e_{g-1,2}+O(q^{\frac{g}{2}(1+\epsilon)}),
\end{align}
where $\mathcal{A}^e_{k,\ell}$ and $\mathcal{B}^e_{k,\ell}$ are the integrals stated at the end of the subsection.
\end{lemma}
\begin{proof}
From (\ref{eq:3.18}) and using the function field analogue of Perron's formula, we obtain
\begin{equation*}
    S^e_1(l^2;f,C)=\frac{1}{2\pi i}\oint_{|z|=q^{-1-\epsilon}}\frac{z^g(q-1)(qz-1)A_f(z)}{q(1-z)z^{\frac{d(f)}{2}+d(C)}}dz.
\end{equation*}
Also, using the fact that (see \cite{Florea2017}, Proof of Lemma 6.1),
\begin{equation*}
    \sum_{\substack{C|f^{\infty}\\C\in\mathbb{A}^+_{\leq g}}}\frac{1}{|C|^2z^{d(C)}}=\prod_{P|f}(1-z^{-d(P)}|P|^{-2})^{-1}+O(q^{g(\epsilon-1)}),
\end{equation*}
we have 
\begin{equation*}
    \mathcal{S}^e_{g,1}(V=\square)=\frac{q^{2g+2}}{2\pi i}\oint_{|z|=q^{-1-\epsilon}}\frac{z^g(q-1)(qz-1)}{q(1-z)}H^e_{g,1}(z)dz+O(q^{\frac{g}{2}(1+\epsilon)}),
\end{equation*}
where 
\begin{equation*}
    H^e_{g,1}(z)=\sum_{\substack{f\in\mathbb{A}^+_{\leq g}\\d(f)\text{ even}}}\frac{A_f(z)}{|f|z^{\frac{d(f)}{2}}}\prod_{P|f}(1-|P|^{-2}z^{-d(P)})^{-1}.
\end{equation*}
Similarly we have 
\begin{equation*}
    \mathcal{S}_{g,2}^e(V=\square)=\frac{q^{\frac{3g+3}{2}}}{2\pi i}\oint_{|z|=q^{-1-\epsilon}}\frac{z^g(q-1)(qz-1)}{q(1-z)}H_{g,2}^e(z)dz+O(q^{\frac{g}{2}(1+\epsilon)}),
\end{equation*}
\begin{equation*}
    \mathcal{S}^e_{g-1,1}(V=\square)=\frac{q^{2g+2}}{2\pi i}\oint_{|z|=q^{-1-\epsilon}}\frac{z^g(q-1)(qz-1)}{q(1-z)}H_{g-1,1}^e(z)dz+O(q^{\frac{g}{2}(1+\epsilon)})
\end{equation*}
and
\begin{equation*}
   \mathcal{S}^e_{g-1,2}(V=\square)=\frac{q^{\frac{3g}{2}+2}}{2\pi i}\oint_{|z|=q^{-1-\epsilon}}\frac{z^g(q-1)(qz-1)}{q(1-z)}H_{g-1,2}^e(z)dz+O(q^{\frac{g}{2}(1+\epsilon)}),
\end{equation*}
where
\begin{equation*}
    H_{g,2}^e(z)=\sum_{\substack{f\in\mathbb{A}^+_{\leq g}\\d(f)\text{ even}}}\frac{A_f(z)}{\sqrt{|f|}z^{\frac{d(f)}{2}}}\prod_{P|f}(1-|P|^{-2}z^{-d(P)})^{-1},
\end{equation*}
\begin{equation*}
    H_{g-1,1}^e(z)=\sum_{\substack{f\in\mathbb{A}^+_{\leq g-1}\\d(f)\text{ even}}}\frac{A_f(z)}{|f|z^{\frac{d(f)}{2}}}\prod_{P|f}(1-|P|^{-2}z^{-d(P)})^{-1}
\end{equation*}
and
\begin{equation*}
    H^e_{g-1,2}(z)=\sum_{\substack{f\in\mathbb{A}^+_{\leq g-1}\\d(f)\text{ even}}}\frac{A_f(z)}{\sqrt{|f|}z^{\frac{d(f)}{2}}}\prod_{P|f}(1-|P|^{-2}z^{-d(P)})^{-1}.
\end{equation*}
Using the function field analogue of Perron's formula, we have
\begin{equation}\label{eq:5.7}
    H^e_{g,1}(z)=\frac{1}{2\pi i}\oint_{|w|=r_2}\frac{\mathcal{B}(z,w)}{w(1-q^2zw^2)(q^2zw^2)^{\left[\frac{g}{2}\right]}}dw-\frac{1}{2\pi i}\oint_{|w|=r_2}\frac{q^2zw\mathcal{B}(z,w)}{1-q^2zw^2}dw,
\end{equation}
\begin{equation}\label{eq:5.8}
    H_{g,2}^e(z)=\frac{1}{2\pi i}\oint_{|w|=r_2}\frac{\mathcal{B}(z,w)}{w(1-qzw^2)(qzw^2)^{\left[\frac{g}{2}\right]}}dw-\frac{1}{2\pi i}\oint_{|w|=r_2}\frac{qzw\mathcal{B}(z,w)}{1-qzw^2}dw,
\end{equation}
\begin{equation}\label{eq:5.9}
    H_{g-1,1}^e(z)=\frac{1}{2\pi i}\oint_{|w|=r_2}\frac{\mathcal{B}(z,w)}{w(1-q^2zw^2)(q^2zw^2)^{\left[\frac{g-1}{2}\right]}}dw-\frac{1}{2\pi i}\oint_{|w|=r_2}\frac{q^2zw\mathcal{B}(z,w)}{1-q^2zw^2}dw
\end{equation}
and
\begin{equation}\label{eq:5.10}
    H_{g-1,2}^e(z)=\frac{1}{2\pi i}\oint_{|w|=r_2}\frac{\mathcal{B}(z,w)}{w(1-qzw^2)(qzw^2)^{\left[\frac{g-1}{2}\right]}}dw-\frac{1}{2\pi i}\oint_{|w|=r_2}\frac{qzw\mathcal{B}(z,w)}{1-qzw^2}dw.
\end{equation}
The second integrals in (\ref{eq:5.7}), (\ref{eq:5.8}), (\ref{eq:5.9}) and (\ref{eq:5.10}) are zero since the integrands have no poles inside the circle $|w|=r_2<q^{-1}$. Therefore we have
\begin{equation*}
    \mathcal{S}_{g,1}^e(V=\square)=\frac{q^{2g+2}}{(2\pi i)^2}\oint_{|z|=q^{-1-\epsilon}}\oint_{|w|=r_2}\frac{z^g(q-1)(qz-1)\mathcal{B}(z,w)}{qw(1-z)(1-q^2zw^2)(q^2zw^2)^{\left[\frac{g}{2}\right]}}dwdz+O(q^{\frac{g}{2}(1+\epsilon)}),
\end{equation*}
\begin{equation*}
    \mathcal{S}^e_{g,2}(V=\square)=\frac{q^{\frac{3g+3}{2}}}{(2\pi i)^2}\oint_{|z|=q^{-1-\epsilon}}\oint_{|w|=r_2}\frac{z^g(q-1)(qz-1)\mathcal{B}(z,w)}{qw(1-z)(1-qzw^2)(qzw^2)^{\left[\frac{g}{2}\right]}}dwdz+O(q^{\frac{g}{2}(1+\epsilon)}),
\end{equation*}
\begin{equation*}
    \mathcal{S}^e_{g-1,1}(V=\square)=\frac{q^{2g+2}}{(2\pi i)^2}\oint_{|z|=q^{-1-\epsilon}}\oint_{|w|=r_2}\frac{z^g(q-1)(qz-1)\mathcal{B}(z,w)}{qw(1-z)(1-q^2zw^2)(q^2zw^2)^{\left[\frac{g-1}{2}\right]}}dwdz+O(q^{\frac{g}{2}(1+\epsilon)})
\end{equation*}
and
\begin{equation*}
    \mathcal{S}_{g-1,2}^e(V=\square)=\frac{q^{\frac{3g}{2}+2}}{(2\pi i)^2}\oint_{|z|=q^{-1-\epsilon}}\oint_{|w|=r_2}\frac{z^g(q-1)(qz-1)\mathcal{B}(z,w)}{qw(1-z)(1-qzw^2)(qzw^2)^{\left[\frac{g-1}{2}\right]}}dwdz+O(q^{\frac{g}{2}(1+\epsilon)}).
\end{equation*}
Using equation (\ref{eq:5.5}) in Lemma 5.2 we obtain
\begin{equation*}
    \mathcal{S}^e_{g,1}(V=\square)=-\frac{q^{2g+2}}{(2\pi i)^2}\oint_{|z|=q^{-1-\epsilon}}\oint_{|w|=r_2}\frac{z^g(q-1)\prod_P\mathcal{B}_P(z,w)}{qw(1-z)(1-qw)(1-q^2zw^2)^2(q^2zw^2)^{\left[\frac{g}{2}\right]}}dwdz+O(q^{\frac{g}{2}(1+\epsilon)}),
\end{equation*}
\begin{align*}
    \mathcal{S}^e_{g,2}(V=\square)&=-\frac{q^{\frac{3g+3}{2}}}{(2\pi i)^2}\oint_{|z|=q^{-1-\epsilon}}\oint_{|w|=r_2}\frac{z^g(q-1)\prod_P\mathcal{B}_P(z,w)}{qw(1-z)(1-qw)(1-qzw^2)(1-q^2zw^2)(qzw^2)^{\left[\frac{g}{2}\right]}}dwdz\\&+O(q^{\frac{g}{2}(1+\epsilon)}),
\end{align*}
\begin{equation*}
    \mathcal{S}_{g-1,1}^e(V=\square)=-\frac{q^{2g+2}}{(2\pi i)^2}\oint_{|z|=q^{-1-\epsilon}}\oint_{|w|=r_2}\frac{z^g(q-1)\prod_P\mathcal{B}_P(z,w)}{qw(1-z)(1-qw)(1-q^2zw^2)^2(q^2zw^2)^{\left[\frac{g-1}{2}\right]}}dwdz+O(q^{\frac{g}{2}(1+\epsilon)})
\end{equation*}
and
\begin{align*}
    \mathcal{S}_{g-1,2}^e(V=\square)&=-\frac{q^{\frac{3g}{2}+2}}{(2\pi i)^2}\oint_{|z|=q^{-1-\epsilon}}\oint_{|w|=r_2}\frac{z^g(q-1)\prod_P\mathcal{B}_P(z,w)}{qw(1-z)(1-qw)(1-qzw^2)(1-q^2zw^2)(qzw^2)^{\left[\frac{g-1}{2}\right]}}dwdz\\&+O(q^{\frac{g}{2}(1+\epsilon)}),
\end{align*}
Using equation (\ref{eq:5.6}) in Lemma 5.3, we obtain
\begin{align*}
    \mathcal{S}^e_{g,1}(V=\square)&=-\frac{q^{2g+2}}{(2\pi i)^2}\oint_{|z|=q^{-1-\epsilon}}\oint_{|w|=r_2}\frac{z^g(q-1)(1-qw^2)\prod_P\mathcal{D}_P(z,w)}{qw(1-z)(1-qw)(1-\frac{w}{qz})(1-q^2zw^2)^2(q^2zw^2)^{\left[\frac{g}{2}\right]}}dwdz\\&+O(q^{\frac{g}{2}(1+\epsilon)}),
\end{align*}
\begin{align*}
    \mathcal{S}^e_{g,2}(V=\square)&=-\frac{q^{\frac{3g+3}{2}}}{(2\pi i)^2}\oint_{|z|=q^{-1-\epsilon}}\oint_{|w|=r_2}\frac{z^g(q-1)(1-qw^2)\prod_P\mathcal{D}_P(z,w)}{qw(1-z)(1-qw)(1-\frac{w}{qz})(1-qzw^2)(1-q^2zw^2)(qzw^2)^{\left[\frac{g}{2}\right]}}dwdz\\&+O(q^{\frac{g}{2}(1+\epsilon)}),
\end{align*}
\begin{align*}
    \mathcal{S}_{g-1,1}^e(V=\square)&=-\frac{q^{2g+2}}{(2\pi i)^2}\oint_{|z|=q^{-1-\epsilon}}\oint_{|w|=r_2}\frac{z^g(q-1)(1-qw^2)\prod_P\mathcal{D}_P(z,w)}{qw(1-z)(1-qw)(1-\frac{w}{qz})(1-q^2zw^2)^2(q^2zw^2)^{\left[\frac{g-1}{2}\right]}}dwdz\\&+O(q^{\frac{g}{2}(1+\epsilon)})
\end{align*}
and
\begin{align*}
    \mathcal{S}_{g-1,2}^e(V=\square)&=-\frac{q^{\frac{3g}{2}+2}}{(2\pi i)^2}\oint_{|z|=q^{-1-\epsilon}}\oint_{|w|=r_2}\frac{z^g(q-1)(1-qw^2)\prod_P\mathcal{D}_P(z,w)}{qw(1-z)(1-qw)(1-\frac{w}{qz})(1-qzw^2)(1-q^2zw^2)(qzw^2)^{\left[\frac{g-1}{2}\right]}}dwdz\\&+O(q^{\frac{g}{2}(1+\epsilon)}),
\end{align*}
Shrinking the contour $|z|=q^{-1-\epsilon}$ to $|z|=q^{-\frac{3}{2}}$, we do not encounter any poles. Enlarging the contour $|w|=r_2<q^{-1}$ to $|w|=q^{-\frac{1}{4}-\epsilon}$, we encounter two simple poles, one at $w=q^{-1}$ and one at $w=qz$. Evaluating the residues at $w=q^{-1}$ and $w=qz$ and writing
\begin{equation}\label{eq:5.11}
    \mathcal{S}^e_{k,\ell}(V=\square)=\mathcal{A}_{k,\ell}^e+\mathcal{B}_{k,\ell}^e+\mathcal{C}_{k,\ell}^e+O(q^{\frac{g}{2}(1+\epsilon)}),
\end{equation}
whilst using Lemma 6.3 where we have for each $k,\ell$, $\mathcal{C}_{k,\ell}^e\ll q^{\frac{g}{2}(1+\epsilon)}$, then we have 
\begin{align*}
    \mathcal{A}^e_{g,1}&=-\frac{q^{2g+2}}{2\pi i}\oint_{|z|=q^{-\frac{3}{2}}}\frac{z^g(q-1)\prod_P\mathcal{B}_P(z,q^{-1})}{q(1-z)^3z^{\left[\frac{g}{2}\right]}}dz,\\
    \mathcal{B}_{g,1}^e&=-\frac{q^{2g+2}}{2\pi i}\oint_{|z|=q^{-\frac{3}{2}}}\frac{z^g(q-1)(1-q^3z^2)\prod_P\mathcal{D}_P(z,qz)}{q(1-z)(1-q^2z)(1-q^4z^3)^2(q^4z^3)^{\left[\frac{g}{2}\right]}}dz,
\end{align*}
\begin{align*}
    \mathcal{A}_{g,2}^e&=-\frac{q^{\frac{3g+3}{2}}}{2\pi i}\oint_{|z|=q^{-\frac{3}{2}}}\frac{z^g(q-1)\prod_P\mathcal{B}_P(z,q^{-1})}{q(1-z)^2(1-q^{-1}z)(q^{-1}z)^{\left[\frac{g}{2}\right]}}dz,\\
    \mathcal{B}^e_{g,2}&=-\frac{q^{\frac{3g+3}{2}}}{2\pi i}\oint_{|z|=q^{-\frac{3}{2}}}\frac{z^g(q-1)(1-q^3z^2)\prod_P\mathcal{D}_P(z,qz)}{q(1-z)(1-q^2z)(1-q^3z^3)(1-q^4z^3)(q^3z^3)^{\left[\frac{g}{2}\right]}}dz,
\end{align*}
\begin{align*}
    \mathcal{A}_{g-1,1}^e&=-\frac{q^{2g+2}}{2\pi i}\oint_{|z|=q^{-\frac{3}{2}}}\frac{z^g(q-1)\prod_P\mathcal{B}_P(z,q^{-1})}{q(1-z)^3z^{\left[\frac{g-1}{2}\right]}}dz,\\
    \mathcal{B}_{g-1,1}^e&=-\frac{q^{2g+2}}{2\pi i}\oint_{|z|=q^{-\frac{3}{2}}}\frac{z^g(q-1)(1-q^3z^2)\prod_P\mathcal{D}_P(z,qz)}{q(1-z)(1-q^2z)(1-q^4z^3)^2(q^4z^3)^{\left[\frac{g-1}{2}\right]}}dz
\end{align*}
and
\begin{align*}
    \mathcal{A}^e_{g-1,2}&=-\frac{q^{\frac{3g}{2}+2}}{2\pi i}\oint_{|z|=q^{-\frac{3}{2}}}\frac{z^g(q-1)\prod_P\mathcal{B}_P(z,q^{-1})}{q(1-z)^2(1-q^{-1}z)(q^{-1}z)^{\left[\frac{g-1}{2}\right]}}dz,\\
    \mathcal{B}_{g-1,2}^e&=\frac{q^{\frac{3g}{2}+2}}{2\pi i}\oint_{|z|=q^{-\frac{3}{2}}}\frac{z^g(q-1)(1-q^3z^2)\prod_P\mathcal{D}_P(z,qz)}{q(1-z)(1-q^2z)(1-q^3z^3)(1-q^4z^3)(q^3z^3)^{\left[\frac{g-1}{2}\right]}}dz.
\end{align*}
\end{proof}
\subsection{Degree $f$ odd}
In this subsection, we prove the following result.
\begin{lemma}
We have
\begin{align}
    \mathcal{S}^o(V=\square)=\mathcal{A}^o_{g,1}-\mathcal{A}^o_{g,2}+\mathcal{A}^o_{g-1,1}-\mathcal{A}^o_{g-1,2}+\mathcal{B}^o_{g,1}-\mathcal{B}^o_{g,2}+\mathcal{B}^o_{g-1,1}-\mathcal{B}^o_{g-1,2}+O(q^{\frac{g}{2}(1+\epsilon)}),
\end{align}
where $\mathcal{A}^o_{k,\ell}$ and $\mathcal{B}^o_{k,\ell}$ are the integrals stated at the end of the subsection.
\end{lemma}
\begin{proof}

From (\ref{eq:3.9}) and using the function field analogue of Perron's formula, we have
\begin{equation*}
    S^o(l^2;f,C)=\frac{1}{2\pi i}\oint_{|z|=q^{-1-\epsilon}}\frac{A_f(z)z^{g-\frac{1}{2}}(qz-1)}{qz^{\frac{d(f)}{2}+d(C)}}dz.
\end{equation*}
Also, using the fact that (see, \cite{Florea2017}, Proof of Lemma 6.1), 
\begin{equation*}
  \sum_{\substack{C|f^{\infty}\\C\in\mathbb{A}^+_{\leq g}}}\frac{1}{|C|^2z^{d(C)}}=\prod_{P|f}(1-z^{-d(P)}|P|^{-2})^{-1}+O(q^{g(\epsilon-1)}),    
\end{equation*}
we have
\begin{equation*}
    \mathcal{S}^o_{g,1}(V=\square)=\frac{q^{2g+\frac{5}{2}}}{2\pi i}\oint_{|z|=q^{-1-\epsilon}}\frac{z^{g-\frac{1}{2}}(qz-1)}{q}H^o_{g,1}(z)dz+O(q^{\frac{g}{2}(1+\epsilon)}),
\end{equation*}
where
\begin{equation*}
    H^o_{g,1}(z)=\sum_{\substack{f\in\mathbb{A}^+_{\leq g}\\d(f)\text{ odd}}}\frac{A_f(z)}{|f|z^{\frac{d(f)}{2}}}\prod_{P|f}(1-|P|^{-2}z^{-d(P)}).
\end{equation*}
Similarly we have
\begin{equation*}
    \mathcal{S}^o_{g,2}(V=\square)=\frac{q^{\frac{3g}{2}+2}}{2\pi i}\oint_{|z|=q^{-1-\epsilon}}\frac{z^{g-\frac{1}{2}}(qz-1)}{q}H^o_{g,2}(z)dz+O(q^{\frac{g}{2}(1+\epsilon)}),
\end{equation*}
\begin{equation*}
    \mathcal{S}_{g-1,1}^o(V=\square)=\frac{q^{2g+\frac{5}{2}}}{2\pi i}\oint_{|z|=q^{-1-\epsilon}}\frac{z^{g-\frac{1}{2}}(qz-1)}{q}H^o_{g-1,1}(z)dz+O(q^{\frac{g}{2}(1+\epsilon)})
\end{equation*}
and
\begin{equation*}
    \mathcal{S}^o_{g-1,2}(V=\square)=\frac{q^{\frac{3g+5}{2}}}{2\pi i}\oint_{|z|=q^{-1-\epsilon}}\frac{z^{g-\frac{1}{2}}(qz-1)}{q}H^o_{g-1,2}(z)dz+O(q^{\frac{g}{2}(1+\epsilon)}),
\end{equation*}
where
\begin{equation*}
    H^o_{g,2}(z)=\sum_{\substack{f\in\mathbb{A}^+_{\leq g}\\d(f)\text{ odd}}}\frac{A_f(z)}{\sqrt{|f|}z^{\frac{d(f)}{2}}}\prod_{P|f}(1-|P|^{-2}z^{-d(P)})^{-1},
\end{equation*}
\begin{equation*}
    H^o_{g-1,1}(z)=\sum_{\substack{f\in\mathbb{A}^+_{\leq g-1}\\d(f)\text{ odd}}}\frac{A_f(z)}{|f|z^{\frac{d(f)}{2}}}\prod_{P|f}(1-|P|^{-2}z^{-d(P)})^{-1}
\end{equation*}
and
\begin{equation*}
    H^o_{g-1,2}(z)=\sum_{\substack{f\in\mathbb{A}^+_{\leq g-1}\\d(f)\text{ odd}}}\frac{A_f(z)}{\sqrt{|f|}z^{\frac{d(f)}{2}}}\prod_{P|f}(1-|P|^{-2}z^{-d(P)})^{-1}.
\end{equation*}
Using the function field analogue of Perron's formula, we have 
\begin{equation}\label{eq:5.12}
    H_{g,1}^o(z)=\frac{1}{2\pi i}\oint_{|w|=r_2}\frac{\mathcal{B}(z,w)}{qz^{\frac{1}{2}}w^2(1-q^2zw^2)(q^2zw^2)^{\left[\frac{g-1}{2}\right]}}dw-\frac{1}{2\pi i}\oint_{|w|=r_2}\frac{qz^{\frac{1}{2}}\mathcal{B}(z,w)}{1-q^2zw^2}dw,
\end{equation}
\begin{equation}\label{eq:5.13}
    H_{g,2}^o(z)=\frac{1}{2\pi i}\oint_{|w|=r_2}\frac{\mathcal{B}(z,w)}{q^{\frac{1}{2}}z^{\frac{1}{2}}w^2(1-qzw^2)(qzw^2)^{\left[\frac{g-1}{2}\right]}}dw-\frac{1}{2\pi i}\oint_{|w|=r_2}\frac{q^{\frac{1}{2}}z^{\frac{1}{2}}\mathcal{B}(z,w)}{1-qzw^2}dw,
\end{equation}
\begin{equation}\label{eq:5.14}
    H^o_{g-1,1}(z)=\frac{1}{2\pi i}\oint_{|w|=r_2}\frac{qz^{\frac{1}{2}}\mathcal{B}(z,w)}{(1-q^2zw^2)(q^2zw^2)^{\left[\frac{g}{2}\right]}}dw-\frac{1}{2\pi i}\oint_{|w|=r_2}\frac{qz^{\frac{1}{2}}\mathcal{B}(z,w)}{1-q^2zw^2}dw
\end{equation}
and
\begin{equation}\label{eq:5.15}
    H_{g-1,2}^o(z)=\frac{1}{2\pi i}\oint_{|w|=r_2}\frac{q^{\frac{1}{2}}z^{\frac{1}{2}}\mathcal{B}(z,w)}{(1-q^2zw^2)(q^2zw^2)^{\left[\frac{g}{2}\right]}}dw-\frac{1}{2\pi i}\oint_{|w|=r_2}\frac{q^{\frac{1}{2}}z^{\frac{1}{2}}\mathcal{B}(z,w)}{1-q^2zw^2}dw.
\end{equation}
The second integrals in (\ref{eq:5.12}), (\ref{eq:5.13}), (\ref{eq:5.14}) and (\ref{eq:5.15}) are zero since the integrands have no poles inside the circle $|w|=r_2<q^{-1}$. Therefore we have
\begin{equation*}
    \mathcal{S}_{g,1}^o(V=\square)=\frac{q^{2g+\frac{5}{2}}}{(2\pi i)^2}\oint_{|z|=q^{-1-\epsilon}}\oint_{|w|=r_2}\frac{z^{g-1}(qz-1)\mathcal{B}(z,w)}{q^2w^2(1-q^2zw^2)(q^2zw^2)^{\left[\frac{g-1}{2}\right]}}dwdz+O(q^{\frac{g}{2}(1+\epsilon)}),
\end{equation*}
\begin{equation*}
    \mathcal{S}_{g,2}^o(V=\square)=\frac{q^{\frac{3g}{2}+2}}{(2\pi i)^2}\oint_{|z|=q^{-1-\epsilon}}\oint_{|w|=r_2}\frac{z^{g-1}(qz-1)\mathcal{B}(z,w)}{q^{\frac{3}{2}}w^2(1-qzw^2)(qzw^2)^{\left[\frac{g-1}{2}\right]}}dwdz+O(q^{\frac{g}{2}(1+\epsilon)}),
\end{equation*}
\begin{equation*}
    \mathcal{S}_{g-1,1}^o(V=\square)=\frac{q^{2g+\frac{5}{2}}}{(2\pi i)^2}\oint_{|z|=q^{-1-\epsilon}}\oint_{|w|=r_2}\frac{z^g(qz-1)\mathcal{B}(z,w)}{(1-q^2zw^2)(q^2zw^2)^{\left[\frac{g}{2}\right]}}dwdz+O(q^{\frac{g}{2}(1+\epsilon)})
\end{equation*}
and
\begin{equation*}
    \mathcal{S}^o_{g-1,2}(V=\square)=\frac{q^{\frac{3g+5}{2}}}{(2\pi i)^2}\oint_{|z|=q^{-1-\epsilon}}\oint_{|w|=r_2}\frac{z^g(qz-1)\mathcal{B}(z,w)}{q^{\frac{1}{2}}(1-qzw^2)(qzw^2)^{\left[\frac{g}{2}\right]}}dwdz+O(q^{\frac{g}{2}(1+\epsilon)}).
\end{equation*}
Using equation (\ref{eq:5.5}) in Lemma 5.2 we have 
\begin{equation*}
    \mathcal{S}^o_{g,1}(V=\square)=-\frac{q^{2g+\frac{5}{2}}}{(2\pi i)^2}\oint_{|z|=q^{-1-\epsilon}}\oint_{|w|=r_2}\frac{z^{g-1}\prod_P\mathcal{B}_P(z,w)}{q^2w^2(1-qw)(1-q^2zw^2)(q^2zw^2)^{\left[\frac{g-1}{2}\right]}}dwdz+O(q^{\frac{g}{2}(1+\epsilon)}),
\end{equation*}
\begin{align*}
    \mathcal{S}^o_{g,2}(V=\square)&=-\frac{q^{\frac{3g}{2}+2}}{(2\pi i)^2}\oint_{|z|=q^{-1-\epsilon}}\oint_{|w|=r_2}\frac{z^{g-1}\prod_P\mathcal{B}_P(z,w)}{q^{\frac{3}{2}}w^2(1-qw)(1-qzw^2)(1-q^2zw^2)(qzw^2)^{\left[\frac{g-1}{2}\right]}}dwdz\\&+O(q^{\frac{g}{2}(1+\epsilon)}),
\end{align*}
\begin{equation*}
    \mathcal{S}^o_{g-1,1}(V=\square)=-\frac{q^{2g+\frac{5}{2}}}{(2\pi i)^2}\oint_{|z|=q^{-1-\epsilon}}\oint_{|w|=r_2}\frac{z^g\prod_P\mathcal{B}_P(z,w)}{(1-qw)(1-q^2zw^2)^2(q^2zw^2)^{\left[\frac{g}{2}\right]}}dwdz+O(q^{\frac{g}{2}(1+\epsilon)})
\end{equation*}
and
\begin{equation*}
    \mathcal{S}^o_{g-1,2}(V=\square)=-\frac{q^{\frac{3g+5}{2}}}{(2\pi i)^2}\oint_{|z|=q^{-1-\epsilon}}\oint_{|w|=r_2}\frac{z^g\prod_P\mathcal{B}_P(z,w)}{q^{\frac{1}{2}}(1-qw)(1-qzw^2)(1-q^2zw^2)(qzw^2)^{\left[\frac{g}{2}\right]}}dwdz+O(q^{\frac{g}{2}(1+\epsilon)}).
\end{equation*}
Using equation (\ref{eq:5.6}) in Lemma 5.3, we have
\begin{align*}
    \mathcal{S}^o_{g,1}(V=\square)&=-\frac{q^{2g+\frac{5}{2}}}{(2\pi i)^2}\oint_{|z|=q^{-1-\epsilon}}\oint_{|w|=r_2}\frac{z^{g-1}(1-qw^2)\prod_P\mathcal{D}_P(z,w)}{q^2w^2(1-qw)(1-\frac{w}{qz})(1-q^2zw^2)(q^2zw^2)^{\left[\frac{g-1}{2}\right]}}dwdz\\&+O(q^{\frac{g}{2}(1+\epsilon)}),
\end{align*}
\begin{align*}
     \mathcal{S}^o_{g,2}(V=\square)&=-\frac{q^{\frac{3g}{2}+2}}{(2\pi i)^2}\oint_{|z|=q^{-1-\epsilon}}\oint_{|w|=r_2}\frac{z^{g-1}(1-qw^2)\prod_P\mathcal{D}_P(z,w)}{q^{\frac{3}{2}}w^2(1-qw)(1-\frac{w}{qz})(1-qzw^2)(1-q^2zw^2)(qzw^2)^{\left[\frac{g-1}{2}\right]}}dwdz\\&+O(q^{\frac{g}{2}(1+\epsilon)}),
\end{align*}
\begin{equation*}
    \mathcal{S}^o_{g-1,1}(V=\square)=-\frac{q^{2g+\frac{5}{2}}}{(2\pi i)^2}\oint_{|z|=q^{-1-\epsilon}}\oint_{|w|=r_2}\frac{z^g(1-qw^2)\prod_P\mathcal{D}_P(z,w)}{(1-qw)(1-\frac{w}{qz})(1-q^2zw^2)^2(q^2zw^2)^{\left[\frac{g}{2}\right]}}dwdz+O(q^{\frac{g}{2}(1+\epsilon)})
\end{equation*}
and
\begin{align*}
     \mathcal{S}^o_{g-1,2}(V=\square)&=-\frac{q^{\frac{3g+5}{2}}}{(2\pi i)^2}\oint_{|z|=q^{-1-\epsilon}}\oint_{|w|=r_2}\frac{z^g(1-qw^2)\prod_P\mathcal{D}_P(z,w)}{q^{\frac{1}{2}}(1-qw)(1-\frac{w}{qz})(1-qzw^2)(1-q^2zw^2)(qzw^2)^{\left[\frac{g}{2}\right]}}dwdz\\&+O(q^{\frac{g}{2}(1+\epsilon)}).
\end{align*}
Shrinking the contour $|z|=q^{-1-\epsilon}$ to $|z|=q^{-\frac{3}{2}}$, we do not encounter any poles. Enlarging the contour $|w|=r_2<q^{-1}$ to $|w|=q^{-\frac{1}{4}-\epsilon}$, we encounter two simple poles, one at $w=q^{-1}$ and one at $w=qz$. Evaluating the residues at $w=q^{-1}$ and $w=qz$ and writing
\begin{equation}\label{eq:5.16}
    \mathcal{S}^o_{k,\ell}(V=\square)=\mathcal{A}_{k,\ell}^o+\mathcal{B}_{k,\ell}^o+\mathcal{C}_{k,\ell}^o+O(q^{\frac{g}{2}(1+\epsilon)}),
\end{equation}
whilst using Lemma 6.3 where we have for each $k,\ell$, $\mathcal{C}_{k,\ell}^o\ll q^{\frac{g}{2}(1+\epsilon)}$, then we have 
\begin{align*}
    \mathcal{A}^o_{g,1}&=-\frac{q^{2g+\frac{5}{2}}}{2\pi i}\oint_{|z|=q^{-\frac{3}{2}}}\frac{z^{g-1}\prod_P\mathcal{B}_P(z,q^{-1})}{q(1-z)^2z^{\left[\frac{g-1}{2}\right]}}dz,\\
    \mathcal{B}_{g,1}^o&=-\frac{q^{2g+\frac{5}{2}}}{2\pi i}\oint_{|z|=q^{-\frac{3}{2}}}\frac{z^{g-2}(1-q^3z^2)\prod_P\mathcal{D}_P(z,qz)}{q^3(1-q^2z)(1-q^4z^3)^2(q^4z^3)^{\left[\frac{g-1}{2}\right]}}dz,
\end{align*}
\begin{align*}
    \mathcal{A}^o_{g,2}&=-\frac{q^{\frac{3g}{2}+2}}{2\pi i}\oint_{|z|=q^{-\frac{3}{2}}}\frac{z^{g-1}\prod_P\mathcal{B}_P(z,q^{-1})}{q^{\frac{1}{2}}(1-z)(1-q^{-1}z)(q^{-1}z)^{\left[\frac{g-1}{2}\right]}}dz,\\
    \mathcal{B}^o_{g,2}&=-\frac{q^{\frac{3g}{2}+2}}{2\pi i}\oint_{|z|=q^{-\frac{3}{2}}}\frac{z^{g-2}(1-q^3z^2)\prod_P\mathcal{D}_P(z,qz)}{q^{\frac{5}{2}}(1-q^2z)(1-q^3z^3)(1-q^4z^3)(q^3z^3)^{\left[\frac{g-1}{2}\right]}}dz,
\end{align*}
\begin{align*}
    \mathcal{A}^o_{g-1,1}&=-\frac{q^{2g+\frac{5}{2}}}{2\pi i}\oint_{|z|=q^{-\frac{3}{2}}}\frac{z^g\prod_P\mathcal{B}_P(z,q^{-1})}{q(1-z)^2z^{\left[\frac{g}{2}\right]}}dz,\\
    \mathcal{B}_{g-1,1}^o&=-\frac{q^{2g+\frac{5}{2}}}{2\pi i}\oint_{|z|=q^{-\frac{3}{2}}}\frac{qz^{g+1}(1-q^3z^2)\prod_P\mathcal{D}_P(z,qz)}{(1-q^2z)(1-q^4z^3)^2(q^4z^3)^{\left[\frac{g}{2}\right]}}dz
\end{align*}
and
\begin{align*}
    \mathcal{A}_{g-1,2}^o&=-\frac{q^{\frac{3g+5}{2}}}{2\pi i}\oint_{|z|=q^{-\frac{3}{2}}}\frac{z^g\prod_P\mathcal{B}_P(z,q^{-1})}{q^{\frac{3}{2}}(1-z)(1-q^{-1}z)(q^{-1}z)^{\left[\frac{g}{2}\right]}}dz,\\
    \mathcal{B}^o_{g-1,2}&=-\frac{q^{\frac{3g+5}{2}}}{2\pi i}\oint_{|z|=q^{-\frac{3}{2}}}\frac{q^{\frac{1}{2}}z^{g+1}(1-q^3z^2)\prod_P\mathcal{D}_P(z,qz)}{(1-q^2z)(1-q^3z^3)(1-q^4z^3)(q^3z^3)^{\left[\frac{g}{2}\right]}}dz. 
\end{align*}
\end{proof}
\subsection{Contribution from $\mathcal{A}$ Terms}
In this subsection, we will focus on evaluating the $\mathcal{A}$ terms which corresponds to the pole at $w=q^{-1}$, these will give the main terms in Proposition 5.1. Let 
\begin{equation*}
    \mathcal{A}=\mathcal{A}^e_{g,1}-\mathcal{A}^e_{g,2}+\mathcal{A}^e_{g-1,1}-\mathcal{A}^e_{g-1,2}+\mathcal{A}^o_{g,1}-\mathcal{A}^o_{g,2}+\mathcal{A}^o_{g-1,1}-\mathcal{A}^o_{g-1,2},
\end{equation*}
then, the main result in this subsection is the following.
\begin{lemma}Using the same notation as before, we have
\begin{equation*}
    \mathcal{A}=\mathcal{S}_1(V=\square)+\mathcal{S}_2(V=\square)+\mathcal{S}_3(V=\square)+\mathcal{S}_4(V=\square)+O(q^{\frac{g}{2}(1+\epsilon)}),
\end{equation*}
where, in particular, the terms $\mathcal{S}_1(V=\square), \mathcal{S}_2(V=\square), \mathcal{S}_3(V=\square)$ and $\mathcal{S}_4(V=\square)$ are the integrals stated in Proposition 5.1.
\end{lemma}
\begin{proof}
Rewrite $\mathcal{A}^e_{g,1}$ and $\mathcal{A}^e_{g-1,1}$ as
\begin{equation*}
    \mathcal{A}^e_{g,1}=-\frac{q^{2g+2}}{2\pi i}\oint_{|z|=q^{-\frac{3}{2}}}\frac{z^g(q-\frac{1}{z}+\frac{1}{z}-1)\prod_P\mathcal{B}_P(z,q^{-1})}{q(1-z)^3z^{\left[\frac{g}{2}\right]}}dz
\end{equation*}
and
\begin{equation*}
    \mathcal{A}^e_{g-1,1}=-\frac{q^{2g+2}}{2\pi i}\oint_{|z|=q^{-\frac{3}{2}}}\frac{z^g(q-\frac{1}{z}+\frac{1}{z}-1)\prod_P\mathcal{B}_P(z,q^{-1})}{q(1-z)^3z^{\left[\frac{g-1}{2}\right]}}dz.
\end{equation*}
Then, for $k\in\{g,g-1\}$, let
\begin{equation*}
    \mathcal{A}^e_{k,1}=\mathcal{A}^e_{k,1,1}+\mathcal{A}^e_{k,1,2}
\end{equation*}
where
\begin{align*}
    \mathcal{A}^e_{g,1,1}&=-\frac{q^{2g+2}}{2\pi i}\oint_{|z|=q^{-\frac{3}{2}}}\frac{z^g(1-\frac{1}{qz})\prod_P\mathcal{B}_P(z,q^{-1})}{(1-z)^3z^{\left[\frac{g}{2}\right]}}dz,\\
    \mathcal{A}^e_{g,1,2}&=-\frac{q^{2g+2}}{2\pi i}\oint_{|z|=q^{-\frac{3}{2}}}\frac{z^{g-1}\prod_P\mathcal{B}_P(z,q^{-1})}{q(1-z)^2z^{\left[\frac{g}{2}\right]}}dz
\end{align*}
and
\begin{align*}
     \mathcal{A}^e_{g-1,1,1}&=-\frac{q^{2g+2}}{2\pi i}\oint_{|z|=q^{-\frac{3}{2}}}\frac{z^g(1-\frac{1}{qz})\prod_P\mathcal{B}_P(z,q^{-1})}{(1-z)^3z^{\left[\frac{g-1}{2}\right]}}dz,\\
    \mathcal{A}^e_{g-1,1,2}&=-\frac{q^{2g+2}}{2\pi i}\oint_{|z|=q^{-\frac{3}{2}}}\frac{z^{g-1}\prod_P\mathcal{B}_P(z,q^{-1})}{q(1-z)^2z^{\left[\frac{g-1}{2}\right]}}dz.
\end{align*}
Using the change of variables $z=(qu)^{-1}$, the contour of integration becomes the circle around the origin $|u|=\sqrt{q}$ and note that (from Lemma 5.2) $\prod_P\mathcal{B}_P(\frac{1}{qu},\frac{1}{q})$ is absolutely convergent for $q^{-1}<|u|<q$. We have 
\begin{equation*}
    \mathcal{A}^e_{g,1,1}=-\frac{q^{2g+2}}{2\pi i}\oint_{|u|=\sqrt{q}}\frac{(1-u)\prod_P\mathcal{B}_P\left(\frac{1}{qu},\frac{1}{q}\right)\left(1-\frac{1}{qu}\right)^{-1}}{u(1-qu)^2(qu)^{\left[\frac{g-1}{2}\right]}}du
\end{equation*}
and
\begin{equation*}
    \mathcal{A}^e_{g-1,1,1}=-\frac{q^{2g+2}}{2\pi i}\oint_{|u|=\sqrt{q}}\frac{(1-u)\prod_P\mathcal{B}_P\left(\frac{1}{qu},\frac{1}{q}\right)\left(1-\frac{1}{qu}\right)^{-1}}{u(1-qu)^2(qu)^{\left[\frac{g}{2}\right]}}du.
\end{equation*}
Using the fact (see \cite{Florea2017}, section 6) that 
\begin{equation}\label{eq:5.17}
    (1-u)\prod_P\mathcal{B}_P\left(\frac{1}{qu},\frac{1}{q}\right)\left(1-\frac{1}{qu}\right)^{-1}=\frac{\mathcal{C}(u)}{\zeta_{\mathbb{A}}(2)},
\end{equation}
we get that 
\begin{equation}\label{eq:5.18}
    \mathcal{A}_{g,1,1}^e=-\frac{q^{2g+2}}{\zeta_{\mathbb{A}}(2)}\frac{1}{2\pi i
    }\oint_{|u|=\sqrt{q}}\frac{\mathcal{C}(u)}{u(1-qu)^2(qu)^{\left[\frac{g-1}{2}\right]}}du
\end{equation}
and
\begin{equation}\label{eq:5.19}
      \mathcal{A}_{g-1,1,1}^e=-\frac{q^{2g+2}}{\zeta_{\mathbb{A}}(2)}\frac{1}{2\pi i
    }\oint_{|u|=\sqrt{q}}\frac{\mathcal{C}(u)}{u(1-qu)^2(qu)^{\left[\frac{g}{2}\right]}}du.
\end{equation}
We see that that (\ref{eq:5.18}) and (\ref{eq:5.19}) are precisely the terms $\mathcal{S}_1(V=\square)$ and $\mathcal{S}_2(V=\square)$ given in the statement of Lemma 5.6. Similarly, using the substitution $z=(qu)^{-1}$ we have
\begin{equation*}
    \mathcal{A}^e_{g,1,2}=-\frac{q^{2g+2}}{2\pi i}\oint_{|u|=\sqrt{q}}\frac{\prod_P\mathcal{B}_P\left(\frac{1}{qu},\frac{1}{q}\right)}{(1-qu)^2(qu)^{\left[\frac{g-1}{2}\right]}}du,
\end{equation*}
\begin{equation*}
    \mathcal{A}^e_{g-1,1,2}=-\frac{q^{2g+2}}{2\pi i}\oint_{|u|=\sqrt{q}}\frac{\prod_P\mathcal{B}_P\left(\frac{1}{qu},\frac{1}{q}\right)}{(1-qu)^2(qu)^{\left[\frac{g}{2}\right]}}du,
\end{equation*}
\begin{equation*}
    \mathcal{A}^o_{g,1}=-\frac{q^{2g+\frac{5}{2}}}{2\pi i}\oint_{|u|=\sqrt{q}}\frac{\prod_P\mathcal{B}_P\left(\frac{1}{qu},\frac{1}{q}\right)}{(1-qu)^2(qu)^{\left[\frac{g}{2}\right]}}du
\end{equation*}
and
\begin{equation*}
    \mathcal{A}^o_{g-1,1}=-\frac{q^{2g+\frac{5}{2}}}{2\pi i}\oint_{|u|=\sqrt{q}}\frac{\prod_P\mathcal{B}_P\left(\frac{1}{qu},\frac{1}{q}\right)}{qu(1-qu)^2(qu)^{\left[\frac{g-1}{2}\right]}}du.
\end{equation*}
Using a variant of (\ref{eq:5.17}) we have 
\begin{equation}
\mathcal{A}^e_{g,1,2}=\frac{q^{2g+2}}{\zeta_{\mathbb{A}}(2)}\frac{1}{2\pi i}\oint_{|u|=\sqrt{q}}\frac{\mathcal{C}(u)}{qu(1-u)(1-qu)(qu)^{\left[\frac{g-1}{2}\right]}}du,
\end{equation}
\begin{equation}
\mathcal{A}^e_{g-1,1,2}=\frac{q^{2g+2}}{\zeta_{\mathbb{A}}(2)}\frac{1}{2\pi i}\oint_{|u|=\sqrt{q}}\frac{\mathcal{C}(u)}{qu(1-u)(1-qu)(qu)^{\left[\frac{g}{2}\right]}}du,
\end{equation}
\begin{equation}
    \mathcal{A}^o_{g,1}=\frac{q^{2g+\frac{5}{2}}}{\zeta_{\mathbb{A}}(2)}\frac{1}{2\pi i}\oint_{|u|=\sqrt{q}}\frac{\mathcal{C}(u)}{qu(1-u)(1-qu)(qu)^{\left[\frac{g}{2}\right]}}du
\end{equation}
and
\begin{equation}
    \mathcal{A}^o_{g-1,1}=\frac{q^{2g+\frac{5}{2}}}{\zeta_{\mathbb{A}}(2)}\frac{1}{2\pi i}\oint_{|u|=\sqrt{q}}\frac{\mathcal{C}(u)}{q^2u^2(1-u)(1-qu)(qu)^{\left[\frac{g-1}{2}\right]}}du.
\end{equation}
Rewrite $\mathcal{A}^o_{g-1,1}$ as
\begin{equation}
    \mathcal{A}^o_{g-1}=\frac{q^{2g+\frac{5}{2}}}{\zeta_{\mathbb{A}}(2)}\frac{1}{2\pi i}\oint_{|u|=\sqrt{q}}\frac{\mathcal{C}(u)(1-qu+qu)}{q^2u^2(1-u)(1-qu)(qu)^{\left[\frac{g-1}{2}\right]}}du.
\end{equation}
Then, we let
\begin{equation*}
    \mathcal{A}^o_{g-1,1}=\mathcal{A}^o_{g-1,1,1}+\mathcal{A}^o_{g-1,1,2},
\end{equation*}
where 
\begin{equation*}
    \mathcal{A}^o_{g-1,1,1}=\frac{q^{2g+\frac{5}{2}}}{\zeta_{\mathbb{A}}(2)}\frac{1}{2\pi i}\oint_{|u|=\sqrt{q}}\frac{\mathcal{C}(u)}{qu(1-u)(1-qu)(qu)^{\left[\frac{g-1}{2}\right]}}du
\end{equation*}
and
\begin{equation*}
    \mathcal{A}^o_{g-1,1,2}=\frac{q^{2g+\frac{5}{2}}}{\zeta_{\mathbb{A}}(2)}\frac{1}{2\pi i}\oint_{|u|=\sqrt{q}}\frac{\mathcal{C}(u)}{(1-u)(qu)^{\left[\frac{g-1}{2}\right]+2}}du.
\end{equation*}
Combining $\mathcal{A}^o_{g,1}$ and $\mathcal{A}^e_{g-1,1,2}$, we have 
\begin{equation*}
    \mathcal{A}^o_{g,1}+\mathcal{A}^e_{g-1,1,2}=\frac{q^{2g+2}}{\zeta_{\mathbb{A}}(2)}\frac{1}{2\pi i}\oint_{|u|=\sqrt{q}}\frac{\mathcal{C}(u)(1+q^{\frac{1}{2}})}{qu(1-u)(1-qu)(qu)^{\left[\frac{g}{2}\right]}}du.
\end{equation*}
Using the fact that (see \cite{Jung2013NoteEnsemble}, Proof of Main Theorem)
\begin{equation}\label{eq:5.25}
    1+q^{\frac{1}{2}}=q^{-\frac{g-1}{2}+\left[\frac{g}{2}\right]}+q^{-\frac{g}{2}+\left[\frac{g-1}{2}\right]+1},
\end{equation}
we have 
\begin{equation*}
     \mathcal{A}^o_{g,1}+\mathcal{A}^e_{g-1,1,2}=\frac{q^{2g+2}}{\zeta_{\mathbb{A}}(2)}\frac{1}{2\pi i}\oint_{|u|=\sqrt{q}}\frac{\mathcal{C}(u)(q^{-\frac{g-1}{2}+\left[\frac{g}{2}\right]}+q^{-\frac{g}{2}+\left[\frac{g-1}{2}\right]+1})}{qu(1-u)(1-qu)(qu)^{\left[\frac{g}{2}\right]}}du.
\end{equation*}
Let 
\begin{equation*}
    \mathcal{A}^o_{g,1}+\mathcal{A}^e_{g-1,1,2}=\hat{\mathcal{A}}_1+\hat{\mathcal{A}}_2,
\end{equation*}
where
\begin{equation}
    \hat{\mathcal{A}}_1=\frac{q^{\frac{3g+3}{2}}}{\zeta_{\mathbb{A}}(2)}\frac{1}{2\pi i}\oint_{|u|=\sqrt{q}}\frac{\mathcal{C}(u)}{u(1-u)(1-qu)u^{\left[\frac{g}{2}\right]}}du
\end{equation}
and
\begin{equation*}
    \hat{\mathcal{A}}_2=\frac{q^{\frac{3g}{2}+3+\left[\frac{g-1}{2}\right]}}{\zeta_{\mathbb{A}}(2)}\frac{1}{2\pi i}\oint_{|u|=\sqrt{q}}\frac{\mathcal{C}(u)}{(1-u)(1-qu)(qu)^{\left[\frac{g}{2}\right]+1}}du.
\end{equation*}
Similarly combining $\mathcal{A}^o_{g-1,1,1}$ and $\mathcal{A}^e_{g,1,2}$ and using (\ref{eq:5.25}), we have
\begin{equation*}
    \mathcal{A}^o_{g-1,1,1}+\mathcal{A}^e_{g,1,2}=\tilde{\mathcal{A}}_1+\tilde{\mathcal{A}}_2,
\end{equation*}
where 
\begin{equation}
    \tilde{\mathcal{A}}_1=\frac{q^{\frac{3g}{2}+2}}{\zeta_{\mathbb{A}}(2)}\frac{1}{2\pi i}\oint_{|u|=\sqrt{q}}\frac{\mathcal{C}(u)}{u(1-u)(1-qu)u^{\left[\frac{g-1}{2}\right]}}du
\end{equation}
and
\begin{equation}
    \tilde{\mathcal{A}}_2=\frac{q^{\frac{3g+5}{2}+\left[\frac{g}{2}\right]}}{\zeta_{\mathbb{A}}(2)}\frac{1}{2\pi i}\oint_{|u|=\sqrt{q}}\frac{\mathcal{C}(u)}{(1-u)(1-qu)(qu)^{\left[\frac{g-1}{2}\right]+1}}du.
\end{equation}
We see that that $\hat{\mathcal{A}}_1$ and $\tilde{\mathcal{A}}_1$ are precisely the terms $\mathcal{S}_3(V=\square)$ and $\mathcal{S}_4(V=\square)$ given in the statement of Lemma 5.6. From (\ref{eq:4.4}), we have that $\mathcal{C}(1)=0$, thus, inside the circle $|u|=\sqrt{q}$, $\mathcal{A}^o_{g-1,1,2}$ has a pole of order $\left[\frac{g-1}{2}\right]+2$ at $u=0$. Using the Residue Theorem we have that 
\begin{equation}
\mathcal{A}^o_{g-1,1,2}=\frac{q^{2g+\frac{1}{2}-\left[\frac{g-1}{2}\right]}}{\zeta_{\mathbb{A}}(2)}\sum_{n=0}^{\left[\frac{g-1}{2}\right]+1}\frac{\mathcal{C}^{(n)}(0)}{n!}.    
\end{equation}
Similarly, inside the circle $|u|=\sqrt{q}$, the integrals $\hat{\mathcal{A}}_2$ and $\tilde{\mathcal{A}}_2$ have a simple pole at $u=q^{-1}$ and a pole at $u=0$ of order $\left[\frac{g}{2}\right]+1$ and $\left[\frac{g-1}{2}\right]+1$ respectively. Thus we have 
\begin{equation*}
    \hat{\mathcal{A}}_2=\frac{q^{\frac{5g}{2}+1-2\left[\frac{g}{2}\right]}}{\zeta_{\mathbb{A}}(2)}\sum_{n=0}^{\left[\frac{g}{2}\right]}\frac{\mathcal{C}^{(n)}(0)}{n!}\sum_{k=0}^{\left[\frac{g}{2}\right]-n}q^k-\frac{q^{\frac{3g}{2}+3+\left[\frac{g-1}{2}\right]}}{\zeta_{\mathbb{A}}(2)(q-1)}
\end{equation*}
and
\begin{equation*}
    \tilde{\mathcal{A}}_2=\frac{q^{\frac{5g+1}{2}-2\left[\frac{g-1}{2}\right]}}{\zeta_{\mathbb{A}}(2)}\sum_{n=0}^{\left[\frac{g-1}{2}\right]}\frac{\mathcal{C}^{(n)}(0)}{n!}\sum_{k=0}^{\left[\frac{g-1}{2}\right]-n}q^k-\frac{q^{\frac{3g+5}{2}+\left[\frac{g}{2}\right]}}{\zeta_{\mathbb{A}}(2)(q-1)}.
\end{equation*}
For the remaining integrals, we rewrite $\mathcal{A}^e_{g,2}$ and $\mathcal{A}^e_{g-1,2}$ as
\begin{equation*}
    \mathcal{A}^e_{k,2}=\mathcal{A}_{k,2,1}^e+\mathcal{A}^e_{k,2,2}
\end{equation*}
where 
\begin{align*}
    \mathcal{A}^e_{g,2,1}&=-\frac{q^{\frac{3g+3}{2}}}{2\pi i}\oint_{|z|=q^{-\frac{3}{2}}}\frac{z^g(1-\frac{1}{qz})\prod_P\mathcal{B}_P(z,q^{-1})}{(1-z)^2(1-q^{-1}z)(q^{-1}z)^{\left[\frac{g}{2}\right]}}dz,\\
    \mathcal{A}^e_{g,2,2}&=-\frac{q^{\frac{3g+3}{2}}}{2\pi i}\oint_{|z|=q^{-\frac{3}{2}}}\frac{z^{g-1}\prod_P\mathcal{B}_P(z,q^{-1})}{q(1-z)(1-q^{-1}z)(q^{-1}z)^{\left[\frac{g}{2}\right]}}dz
\end{align*}
and
\begin{align*}
     \mathcal{A}^e_{g-1,2,1}&=-\frac{q^{\frac{3g}{2}+2}}{2\pi i}\oint_{|z|=q^{-\frac{3}{2}}}\frac{z^g(1-\frac{1}{qz})\prod_P\mathcal{B}_P(z,q^{-1})}{(1-z)^2(1-q^{-1}z)(q^{-1}z)^{\left[\frac{g-1}{2}\right]}}dz,\\
    \mathcal{A}^e_{g-1,2,2}&=-\frac{q^{\frac{3g}{2}+2}}{2\pi i}\oint_{|z|=q^{-\frac{3}{2}}}\frac{z^{g-1}\prod_P\mathcal{B}_P(z,q^{-1})}{q(1-z)(1-q^{-1}z)(q^{-1}z)^{\left[\frac{g-1}{2}\right]}}dz.
\end{align*}
Using the substitution $z=(qu)^{-1}$ we have
\begin{equation*}
    \mathcal{A}^e_{g,2,1}=\frac{q^{\frac{5g+5}{2}}}{2\pi i}\oint_{|u|=\sqrt{q}}\frac{(1-u)\prod_P\mathcal{B}_P\left(\frac{1}{qu},\frac{1}{q}\right)}{(1-qu)^2(1-q^2u)(q^2u)^{\left[\frac{g-1}{2}\right]}}du,
\end{equation*}
\begin{equation*}
    \mathcal{A}^e_{g,2,2}=-\frac{q^{\frac{5g+3}{2}}}{2\pi i}\oint_{|u|=\sqrt{q}}\frac{\prod_P\mathcal{B}_P\left(\frac{1}{qu},\frac{1}{q}\right)}{(1-qu)(1-q^2u)(q^2u)^{\left[\frac{g-1}{2}\right]}}du,
\end{equation*}
\begin{equation*}
    \mathcal{A}^e_{g-1,2,1}=\frac{q^{\frac{5g}{2}+3}}{2\pi i}\oint_{|u|=\sqrt{q}}\frac{(1-u)\prod_P\mathcal{B}_P\left(\frac{1}{qu},\frac{1}{q}\right)}{(1-qu)^2(1-q^2u)(q^2u)^{\left[\frac{g}{2}\right]}}du,
\end{equation*}
\begin{equation*}
 \mathcal{A}^e_{g-1,2,2}=-\frac{q^{\frac{5g}{2}+2}}{2\pi i}\oint_{|u|=\sqrt{q}}\frac{\prod_P\mathcal{B}_P\left(\frac{1}{qu},\frac{1}{q}\right)}{(1-qu)(1-q^2u)(q^2u)^{\left[\frac{g}{2}\right]}}du,
\end{equation*}
\begin{equation*}
    \mathcal{A}^o_{g,2}=-\frac{q^{\frac{5g+5}{2}}}{2\pi i}\oint_{|u|=\sqrt{q}}\frac{\prod_P\mathcal{B}_P\left(\frac{1}{qu},\frac{1}{q}\right)}{(1-qu)(1-q^2u)(q^2u)^{\left[\frac{g}{2}\right]}}du
\end{equation*}
and
\begin{equation*}
    \mathcal{A}^o_{g-1,2}=-\frac{q^{\frac{5g}{2}+1}}{2\pi i}\oint_{|u|=\sqrt{q}}\frac{\prod_P\mathcal{B}_P\left(\frac{1}{qu},\frac{1}{q}\right)}{u(1-qu)(1-q^2u)(q^2u)^{\left[\frac{g-1}{2}\right]}}du.
\end{equation*}
Using a variant of (\ref{eq:5.17}), we have
\begin{equation}
  \mathcal{A}^e_{g,2,1}=- \frac{q^{\frac{5g+7}{2}}}{\zeta_{\mathbb{A}}(2)}\frac{1}{2\pi i}\oint_{|u|=\sqrt{q}}\frac{\mathcal{C}(u)}{(1-qu)(1-q^2u)(q^2u)^{\left[\frac{g-1}{2}\right]+1}}du,
\end{equation}
\begin{equation}
   \mathcal{A}^e_{g,2,2}=\frac{q^{\frac{5g+5}{2}}}{\zeta_{\mathbb{A}}(2)}\frac{1}{2\pi i}\oint_{|u|=\sqrt{q}}\frac{\mathcal{C}(u)}{(1-u)(1-q^2u)(q^2u)^{\left[\frac{g-1}{2}\right]+1}}du,
\end{equation}
\begin{equation}
    \mathcal{A}^e_{g-1,2,1}=-\frac{q^{\frac{5g}{2}+4}}{\zeta_{\mathbb{A}}(2)}\frac{1}{2\pi i}\oint_{|u|=\sqrt{q}}\frac{\mathcal{C}(u)}{(1-qu)(1-q^2u)(q^2u)^{\left[\frac{g}{2}\right]+1}}du,
\end{equation}
\begin{equation}
    \mathcal{A}^e_{g-1,2,2}=\frac{q^{\frac{5g}{2}+3}}{\zeta_{\mathbb{A}}(2)}\frac{1}{2\pi i}\oint_{|u|=\sqrt{q}}\frac{\mathcal{C}(u)}{(1-u)(1-q^2u)(q^2u)^{\left[\frac{g}{2}\right]+1}}du,
\end{equation}
\begin{equation}
    \mathcal{A}_{g,2}^o=\frac{q^{\frac{5g+7}{2}}}{\zeta_{\mathbb{A}}(2)}\frac{1}{2\pi i}\oint_{|u|=\sqrt{q}}\frac{\mathcal{C}(u)}{(1-u)(1-q^2u)(q^2u)^{\left[\frac{g}{2}\right]+1}}du
\end{equation}
and
\begin{equation}
    \mathcal{A}_{g-1,2}^o=\frac{q^{\frac{5g}{2}+4}}{\zeta_{\mathbb{A}}(2)}\frac{1}{2\pi i}\oint_{|u|=\sqrt{q}}\frac{\mathcal{C}(u)}{(1-u)(1-q^2u)(q^2u)^{\left[\frac{g-1}{2}\right]+2}}du.
\end{equation}
Inside the circle $|u|=\sqrt{q}$, the integrals all have poles at $u=0,u=q^{-1}$ and $u=q^{-2}$ of varying orders, thus using the Residue Theorem, we have that
\begin{equation}
    \mathcal{A}^e_{g,2,1}=- \frac{q^{\frac{5g+3}{2}-2\left[\frac{g-1}{2}\right]}}{\zeta_{\mathbb{A}}(2)} \sum_{n=0}^{\left[\frac{g-1}{2}\right]}\frac{\mathcal{C}^{(n)}(0)}{n!}\sum_{k=(\left[\frac{g-1}{2}\right]-n)}^{2(\left[\frac{g-1}{2}\right]-n)}q^k-\frac{q^{\frac{5g+3}{2}-\left[\frac{g-1}{2}\right]}}{\zeta_{\mathbb{A}}(2)}\frac{\mathcal{C}(q^{-1})}{q-1}+\frac{q^{\frac{5g+5}{2}}}{\zeta_{\mathbb{A}}(2)}\frac{\mathcal{C}(q^{-2})}{q-1},
\end{equation}
\begin{equation}
    \mathcal{A}_{g,2,2}^e=\frac{q^{\frac{5g+1}{2}-2\left[\frac{g-1}{2}\right]}}{\zeta_{\mathbb{A}}(2)}\sum_{n=0}^{\left[\frac{g-1}{2}\right]}\frac{\mathcal{C}^{(n)}(0)}{n!}\sum_{k=0}^{\left[\frac{g-1}{2}\right]-n}q^{2k}-\frac{q^{\frac{5g+5}{2}}}{\zeta_{\mathbb{A}}(2)}\frac{\mathcal{C}(q^{-2})}{q^2-1},
\end{equation}
\begin{equation}
    \mathcal{A}^e_{g-1,2,1}=- \frac{q^{\frac{5g}{2}+2-2\left[\frac{g}{2}\right]}}{\zeta_{\mathbb{A}}(2)}\sum_{n=0}^{\left[\frac{g}{2}\right]}\frac{\mathcal{C}^{(n)}(0)}{n!}\sum_{k=(\left[\frac{g}{2}\right]-n)}^{2(\left[\frac{g}{2}\right]-n)}q^k-\frac{q^{\frac{5g}{2}+2-\left[\frac{g}{2}\right]}}{\zeta_{\mathbb{A}}(2)}\frac{\mathcal{C}(q^{-1})}{q-1}+\frac{q^{\frac{5g}{2}+3}}{\zeta_{\mathbb{A}}(2)}\frac{\mathcal{C}(q^{-2})}{q-1},
\end{equation}
\begin{equation}
    \mathcal{A}^e_{g-1,2,2}=\frac{q^{\frac{5g}{2}+1-2\left[\frac{g}{2}\right]}}{\zeta_{\mathbb{A}}(2)}\sum_{n=0}^{\left[\frac{g}{2}\right]}\frac{\mathcal{C}^{(n)}(0)}{n!}\sum_{k=0}^{\left[\frac{g}{2}\right]-n}q^{2k}-\frac{q^{\frac{5g}{2}+3}}{\zeta_{\mathbb{A}}(2)}\frac{\mathcal{C}(q^{-2})}{q^2-1},
\end{equation}
\begin{equation}
    \mathcal{A}^o_{g,2}=\frac{q^{\frac{5g+3}{2}-2\left[\frac{g}{2}\right]}}{\zeta_{\mathbb{A}}(2)}\sum_{n=0}^{\left[\frac{g}{2}\right]}\frac{\mathcal{C}^{(n)}(0)}{n!}\sum_{k=0}^{\left[\frac{g}{2}\right]-n}q^{2k}-\frac{q^{\frac{5g+7}{2}}}{\zeta_{\mathbb{A}}(2)}\frac{\mathcal{C}(q^{-2})}{q^2-1},
\end{equation}
\begin{equation}
    \mathcal{A}^o_{g-1,2}=\frac{q^{\frac{5g}{2}-2\left[\frac{g-1}{2}\right]}}{\zeta_{\mathbb{A}}(2)}\sum_{n=0}^{\left[\frac{g-1}{2}\right]+1}\frac{\mathcal{C}^{(n)}(0)}{n!}\sum_{k=0}^{\left[\frac{g-1}{2}\right]+1-n}q^{2k}-\frac{q^{\frac{5g}{2}+4}}{\zeta_{\mathbb{A}}(2)}\frac{\mathcal{C}(q^{-2})}{q^2-1}
\end{equation}
To complete the proof, we want to show that 
\begin{equation}\label{eq:5.44}
    \mathcal{A}^o_{g-1,1,2}+\hat{\mathcal{A}}_2+\tilde{\mathcal{A}}_2-\mathcal{A}^e_{g,2,1}-\mathcal{A}^e_{g,2,2}-\mathcal{A}^e_{g-1,2,1}-\mathcal{A}^e_{g-1,2,2}-\mathcal{A}^o_{g,2}-\mathcal{A}^o_{g-1,2}, 
\end{equation}
equals zero. Using the fact that (see \cite{Jung2014AEnsemble}, section 1)
\begin{equation}
   q^{g-\left[\frac{g-1}{2}\right]}-q^{\left[\frac{g}{2}\right]+1}=0 \text{ and } q^{g-\left[\frac{g}{2}\right]}-q^{\left[\frac{g-1}{2}\right]+1}=0,
\end{equation}
we see that the terms corresponding to residue at $u=q^{-1}$ and $u=q^{-2}$ equal zero. Finally, we use induction on $g$ (see appendix) to show that the terms corresponding to the residue at $u=0$ equals zero. Thus (\ref{eq:5.44}) equals zero. 
\end{proof}
\subsection{Contribution from $\mathcal{B}$ Terms}
We will now focus on evaluating the $\mathcal{B}$ terms which corresponds to the pole at $w=qz$, these will give the secondary main terms in Proposition 5.1. Let
\begin{equation*}
    \mathcal{B}=\mathcal{B}^e_{g,1}-\mathcal{B}^e_{g,2}+\mathcal{B}^e_{g-1,1}-\mathcal{B}^e_{g-1,2}+\mathcal{B}^o_{g,1}-\mathcal{B}^o_{g,2}+\mathcal{B}^o_{g-1,1}-\mathcal{B}^o_{g-1,2},
\end{equation*}
then, the main result in this subsection is the following.
\begin{lemma}
Using the same notation as before, we have that
\begin{equation*}
    \mathcal{B}=q^{\frac{2g+2}{3}}\mathcal{R}(2g+2)+C_1q^{\frac{g}{6}+\left[\frac{g}{2}\right]}+C_2q^{\frac{g}{6}+\left[\frac{g-1}{2}\right]}+O(q^{\frac{g}{2}(1+\epsilon)}),
\end{equation*}
where $\mathcal{R}$ is a polynomial of degree 1 given by (\ref{eq:5.35}) and $C_1$ and $C_2$ are constants given by (\ref{eq:5.36}) and (\ref{eq:5.37}) respectively. 
\end{lemma}
\begin{proof}

For each $j\in\{o,e\}$ and $k\in\{g,g-1\}$ we write 
\begin{equation*}
    \mathcal{B}^j_{k,1}=-\frac{q^{2g+2}}{2\pi i}\oint_{|z|=q^{-\frac{3}{2}}}F^j_{k,1}(z)dz\hspace{0.5cm}\text{and}\hspace{0.5cm}\mathcal{B}^j_{k,2}=-\frac{q^{\frac{3g}{2}+2}}{2\pi i}\oint_{|z|=q^{-\frac{3}{2}}}F^j_{k,2}(z)dz. 
\end{equation*}
Enlarging the contour $|z|=q^{-\frac{3}{2}}$ to $|z|=q^{-1-\epsilon}$ we encounter a double pole at $z=q^{-\frac{4}{3}}$ of $F^j_{k,1}(z)$ and a simple pole at $z=q^{-\frac{4}{3}}$ of $F^j_{k,2}(z)$. From Lemma 5.3, $\prod_P\mathcal{D}_P(z,qz)$ is absolutely convergent when $q^{-2}<|z|<q^{-1}$. Then, we have
\begin{equation*}
    \mathcal{B}^j_{k,1}=q^{2g+2}\text{Res}(F^j_{k,1}(z);z=q^{-\frac{4}{3}})-\frac{q^{2g+2}}{2\pi i}\oint_{|z|=q^{-1-\epsilon}}F^j_{k,1}(z)dz
\end{equation*}
and 
\begin{equation*}
    \mathcal{B}^j_{k,2}=q^{\frac{3g}{2}+2}\text{Res}(F^j_{k,2}(z);z=q^{-\frac{4}{3}})-\frac{q^{\frac{3g}{2}+2}}{2\pi i}\oint_{|z|=q^{-1-\epsilon}}F^j_{k,2}(z)dz
\end{equation*}
where the second terms are bounded by $O(q^{\frac{g}{2}(1+\epsilon)})$. Then
\begin{align*}
    \mathcal{B}_{g,1}^e&=q^{\frac{2g+2}{3}}R_1(g)+O(q^{\frac{g}{2}(1+\epsilon)}), \hspace{0.8cm} \mathcal{B}^e_{g-1,1}=q^{\frac{2g+2}{3}}R_2(g)+O(q^{\frac{g}{2}(1+\epsilon)}),\\
    \mathcal{B}_{g,1}^o&=q^{\frac{2g+2}{3}}R_3(g)+O(q^{\frac{g}{2}(1+\epsilon)})\hspace{0.2cm}\text{and}\hspace{0.2cm}\mathcal{B}_{g-1,1}^o=q^{\frac{2g+2}{3}}R_4(g)+O(q^{\frac{g}{2}(1+\epsilon)}).
\end{align*}
where each $R_i$ is a linear polynomial whose coefficients can be computed explicitly. Let
\begin{equation*}
    q^{\frac{2g+2}{3}}\mathcal{R}(2g+2)=q^{\frac{2g+2}{3}}R_1(g)+q^{\frac{2g+2}{3}}R_2(g)+q^{\frac{2g+2}{3}}R_3(g)+q^{\frac{2g+2}{3}}R_4(g),
\end{equation*}
where
\begin{equation}\label{eq:5.35}
    \mathcal{R}(x)=\frac{\zeta_{\mathbb{A}}\left(\frac{5}{3}\right)\zeta_{\mathbb{A}}\left(\frac{7}{3}\right)}{9q^{\frac{4}{3}}\zeta_{\mathbb{A}}\left(\frac{4}{3}\right)}\prod_P\mathcal{D}_P(q^{-\frac{4}{3}},q^{-\frac{1}{3}})\left[\frac{x}{2}C_3-C_4-\frac{2C_3}{q^{\frac{4}{3}}}\frac{\frac{d}{dz}\prod_P\mathcal{D}_P(z,qz)}{\prod_P\mathcal{D}_P(z,qz)}|_{z=q^{-\frac{4}{3}}}\right],
\end{equation}
\begin{equation*}
        C_3=1-q-q^{\frac{7}{6}}+q^{-\frac{1}{6}}
\end{equation*}
and
\begin{equation*}
    C_4=4C_3\zeta_{\mathbb{A}}\left(\frac{4}{3}\right)-C_3\zeta_{\mathbb{A}}\left(\frac{5}{3}\right)+\frac{2(q-1)\zeta_{\mathbb{A}}\left(\frac{7}{3}\right)}{q^{\frac{4}{3}}}+4(q-1)+2q^{\frac{1}{6}}\zeta_{\mathbb{A}}\left(\frac{7}{3}\right)(1+q).
\end{equation*}
Also we write
\begin{align*}
    \mathcal{B}_{g,2}^e&=-q^{\frac{g}{6}+\left[\frac{g}{2}\right]}C^e_1+O(q^{\frac{g}{2}(1+\epsilon)}), \hspace{0.8cm} \mathcal{B}^e_{g-1,2}=-q^{\frac{g}{6}+\left[\frac{g-1}{2}\right]}C^e_2+O(q^{\frac{g}{2}(1+\epsilon)}),\\
    \mathcal{B}_{g,2}^o&=-q^{\frac{g}{6}+\left[\frac{g-1}{2}\right]}C^o_1+O(q^{\frac{g}{2}(1+\epsilon)})\hspace{0.2cm}\text{and}\hspace{0.2cm}\mathcal{B}_{g-1,2}^o=-q^{\frac{g}{6}+\left[\frac{g}{2}\right]}C^o_2+O(q^{\frac{g}{2}(1+\epsilon)}).
\end{align*}
where each $C^j_{\ell}$ are constants that can be explicitly computed. Let
\begin{equation*}
    C_1=C_1^e+C^o_2\hspace{1cm}\text{and}\hspace{1cm}C_2=C_1^o+C_2^e,
\end{equation*}
then we have
\begin{equation}\label{eq:5.36}
    C_1=\frac{\zeta_{\mathbb{A}}\left(\frac{5}{3}\right)\zeta_{\mathbb{A}}\left(\frac{7}{3}\right)\zeta_{\mathbb{A}}(2)}{\zeta_{\mathbb{A}}\left(\frac{4}{3}\right)}(q^{-\frac{1}{6}}-q^{-\frac{7}{6}}+q^{-\frac{4}{3}}-1)\prod_P\mathcal{D}_P(q^{-\frac{4}{3}},q^{-\frac{1}{3}})
\end{equation}
and
\begin{equation}\label{eq:5.37}
    C_2=\frac{\zeta_{\mathbb{A}}\left(\frac{5}{3}\right)\zeta_{\mathbb{A}}\left(\frac{7}{3}\right)\zeta_{\mathbb{A}}(2)}{\zeta_{\mathbb{A}}\left(\frac{4}{3}\right)}(q^{\frac{1}{3}}-q^{-\frac{2}{3}}+q^{\frac{11}{6}}-q)\prod_P\mathcal{D}_P(q^{-\frac{4}{3}},q^{-\frac{1}{3}}).
\end{equation}
Moreover,
\begin{equation*}
    \prod_P\mathcal{D}_P(q^{-\frac{4}{3}},q^{-\frac{1}{3}})=\prod_P\left(1-\frac{|P|^{\frac{4}{3}}+|P|^{\frac{2}{3}}+|P|^{\frac{1}{3}}+1}{(|P|^{\frac{4}{3}}+|P|)^2}\right)
\end{equation*}
and 
\begin{equation*}
    \frac{1}{q^{\frac{4}{3}}}\frac{\frac{d}{dz}\prod_P\mathcal{D}_P(z,qz)}{\prod_P\mathcal{D}_P(z,qz)}|_{|z|=q^{-\frac{4}{3}}}=-\sum_{P}\frac{d(P)(|P|-1)(|P|^{\frac{1}{3}}+1)}{(|P|^{\frac{1}{3}}-1)(|P|^{\frac{4}{3}}+|P|)^2}.
\end{equation*}
\end{proof}
\par\noindent
Proposition 5.1 is immediate from Lemma 5.4, Lemma 5.5, Lemma 5.6, Lemma 5.7 and (\ref{eq:5.1}). 
\section{Error From Non-Square $V$}
Let 
\begin{equation}\label{eq:6.1}
    \mathcal{S}(V\neq\square)=\mathcal{S}^o(V\neq\square)+\mathcal{S}^e(V\neq\square),
\end{equation}
where
\begin{equation}\label{eq:6.2}
    \mathcal{S}^o(V\neq\square)=\mathcal{S}^o_{g,1}(V\neq\square)-\mathcal{S}^o_{g,2}(V\neq\square)+\mathcal{S}^o_{g-1,1}(V\neq\square)-\mathcal{S}^o_{g-1,2}(V\neq\square)
\end{equation}
and
\begin{equation}\label{eq:6.3}
      \mathcal{S}^e(V\neq\square)=\mathcal{S}^e_{g,1}(V\neq\square)-\mathcal{S}^e_{g,2}(V\neq\square)+\mathcal{S}^e_{g-1,1}(V\neq\square)-\mathcal{S}^e_{g-1,2}(V\neq\square).
\end{equation}
Then, in this section we will bound the term $\mathcal{S}(V\neq\square)$. The next Proposition is the main result in this section.
\begin{prop}
Using the notation described previously, we have, for any $\epsilon>0$
\begin{equation}
    \mathcal{S}(V\neq\square)\ll q^{\frac{g}{2}(1+\epsilon)}.
\end{equation}
\end{prop}
\par\noindent
To prove Proposition 6.1, we will need the following results (see \cite{Florea2017}, section 7). We have 
\begin{equation}\label{eq:6.5}
    \sum_{\substack{C|f^{\infty}\\C\in\mathbb{A}^+_m}}\frac{1}{|C|^2}=\frac{1}{2\pi i}\oint_{|u|=r_1}\frac{1}{q^{2m}u^{m+1}\prod_{P|f}(1-u^{d(P)})}du
\end{equation}
with $r_1<1$. For a non-square $V\in\mathbb{A}^+$ and positive integer $n$, let
\begin{equation*}
    \delta_{V;n}(u)=\sum_{f\in\mathbb{A}^+_n}\frac{G(V,\chi_f)}{\sqrt{|f|}\prod_{P|f}(1-u^{d(P)})}.
\end{equation*}
Then, if $|u|=q^{-\epsilon}$, then we have 
\begin{equation}\label{eq:6.6}
    |\delta_{V;n}(u)|\ll q^{\frac{n}{2}(1+\epsilon)}.
\end{equation}
\subsection{Bounding $\mathcal{S}^e(V\neq\square)$}
For each $k\in\{g,g-1\}$ and $\ell\in\{1,2\}$, we have
\begin{equation*}
    \mathcal{S}^e_{k,\ell}(V\neq\square)=\mathcal{S}^e_{k,\ell,1}(V\neq\square)+\mathcal{S}^e_{k,\ell,2}(V\neq\square).
\end{equation*}
Write 
\begin{equation*}
    \mathcal{S}^e_{k,\ell,1}(V\neq\square)=\tilde{\mathcal{S}}^e_{k,\ell,1}(V\neq\square)-\hat{\mathcal{S}}^e_{k,\ell,1}(V\neq\square)
\end{equation*}
and
\begin{equation*}
    \mathcal{S}^e_{k,\ell,2}(V\neq\square)=\tilde{\mathcal{S}}^e_{k,\ell,2}(V\neq\square)-\hat{\mathcal{S}}^e_{k,\ell,2}(V\neq\square),
\end{equation*}
where $\tilde{\mathcal{S}}^e_{k,\ell,1}(V\neq\square)$, and $\hat{\mathcal{S}}^e_{k,\ell,1}(V\neq\square)$ denote the sum over non-square $V$ of degree $d(V)\leq d(f)-2g-4+2d(C)$ and $d(V)\leq d(f)-2g-2+2d(C)$ respectively. Similarly $\tilde{\mathcal{S}}^e_{k,\ell,2}(V\neq\square)$ and  $\hat{\mathcal{S}}^e_{k,\ell,2}(V\neq\square)$ denote the sum over $V$ with $d(V)=d(f)-2g-1+2d(C)$ and $d(V)=d(f)-2g-3+2d(C)$ respectively. Then, by (\ref{eq:6.5}), we can write
\begin{align*}
    \tilde{\mathcal{S}}^e_{g,1,1}(V\neq\square)&=\frac{(q-1)q^{2g+2}}{2\pi i}\oint_{|u|=r_1}\sum_{n=0}^{\left[\frac{g}{2}\right]}\frac{1}{q^{2k}}\sum_{m=g-n+2}^g\frac{1}{q^{2m}u^{m+1}}\sum_{\square\neq V\in\mathbb{A}^+_{\leq 2n-2g-4+2d(C)}}\delta_{V;2n}(u)du,\\
    \hat{\mathcal{S}}^e_{g,1,1}(V\neq\square)&=\frac{(q-1)q^{2g+1}}{2\pi i}\oint_{|u|=r_1}\sum_{n=0}^{\left[\frac{g}{2}\right]}\frac{1}{q^{2k}}\sum_{m=g-n+1}^g\frac{1}{q^{2m}u^{m+1}}\sum_{\square\neq V\in\mathbb{A}^+_{\leq 2n-2g-2+2d(C)}}\delta_{V;2n}(u)du
\end{align*}
and 
\begin{align*}
    \tilde{\mathcal{S}}^e_{g,1,2}(V\neq\square)&=\frac{q^{2g+1}}{2\pi i}\oint_{|u|=r_1}\sum_{n=0}^{\left[\frac{g}{2}\right]}\frac{1}{q^{2k}}\sum_{m=g-n+1}^g\frac{1}{q^{2m}u^{m+1}}\sum_{V\in\mathbb{A}^+_{2n-2g-1+2m}}\delta_{V;2n}(u)du,\\
      \hat{\mathcal{S}}^e_{g,1,2}(V\neq\square)&=\frac{q^{2g+2}}{2\pi i}\oint_{|u|=r_1}\sum_{n=0}^{\left[\frac{g}{2}\right]}\frac{1}{q^{2k}}\sum_{m=g-n+2}^g\frac{1}{q^{2m}u^{m+1}}\sum_{V\in\mathbb{A}^+_{2n-2g-3+2m}}\delta_{V;2n}(u)du,
\end{align*}
with $r_1<1$. Using (\ref{eq:6.6}), we can bound $\delta_{V;2n}(u)$ and trivially bounding the sum over $V$, we get that $\tilde{\mathcal{S}}^e_{g,1,1}(V\neq\square)\ll q^{\frac{g}{2}(1+\epsilon)}$, $\hat{\mathcal{S}}^e_{g,1,1}(V\neq\square)\ll q^{\frac{g}{2}(1+\epsilon)}$, $\tilde{\mathcal{S}}^e_{g,1,2}(V\neq\square)\ll q^{\frac{g}{2}(1+\epsilon)}$ and $\hat{\mathcal{S}}^e_{g,1,2}(V\neq\square)\ll q^{\frac{g}{2}(1+\epsilon)}$, thus $\mathcal{S}^e_{g,1}(V\neq\square)\ll q^{\frac{g}{2}(1+\epsilon)}$. Using the same calculations, we can bound $\mathcal{S}^e_{g,2}(V\neq\square), \mathcal{S}^e_{g-1,1}(V\neq\square)$ and $\mathcal{S}^e_{g-1,2}(V\neq\square)$ by $q^{\frac{g}{2}(1+\epsilon)}$, hence $\mathcal{S}^e(V\neq\square)\ll q^{\frac{g}{2}(1+\epsilon)}$  .
\subsection{Bounding $\mathcal{S}^o(V\neq\square)$} 
For $k\in\{g,g-1\}, \ell\in\{1,2\}$, let
\begin{equation*}
    \mathcal{S}^o_{k,\ell}(V\neq\square)=\tilde{\mathcal{S}}^o_{k,\ell}(V\neq\square)-\hat{\mathcal{S}}^o_{k,\ell}(V\neq\square),
\end{equation*}
where $\tilde{\mathcal{S}}^o_{k,\ell}$ and $\hat{\mathcal{S}}^o_{k,\ell}$ denotes the sum over non-square $V$ with $d(V)=d(f)-2g-3+2d(C)$ and $d(V)=d(f)-2g-1+2d(C)$ respectively. Then, using (\ref{eq:6.5}), we have 
\begin{equation*}
    \tilde{\mathcal{S}}^o_{g,1}(V\neq\square)=\frac{q^{2g+\frac{5}{2}}}{2\pi i}\oint_{|u|=r_1}\sum_{n=0}^{\left[\frac{g-1}{2}\right]}\frac{1}{q^{2n+1}}\sum_{m=g-n+1}^g\frac{1}{q^{2m}u^{m+1}}\sum_{\square\neq V\in\mathbb{A}^+_{2n-2g-2+2m}}\delta_{V;2n+1}(u)du
\end{equation*}
and
\begin{equation*}
     \hat{\mathcal{S}}^o_{g,1}(V\neq\square)=\frac{q^{2g+\frac{3}{2}}}{2\pi i}\oint_{|u|=r_1}\sum_{n=0}^{\left[\frac{g-1}{2}\right]}\frac{1}{q^{2n+1}}\sum_{m=g-n}^g\frac{1}{q^{2m}u^{m+1}}\sum_{\square\neq V\in\mathbb{A}^+_{2n-2g+2m}}\delta_{V;2n+1}(u)du
\end{equation*}
with $r_1<1$. Using (\ref{eq:6.6}) to bound $\delta_{V;2n+1}(u)$ and trivially bounding the sum over $V$, we get that $\hat{\mathcal{S}}^o_{g,1}(V\neq\square)\ll q^{\frac{g}{2}(1+\epsilon)}$ and $\tilde{\mathcal{S}}^o_{g,1}(V\neq\square)\ll q^{\frac{g}{2}(1+\epsilon)}$. Thus $\mathcal{S}^o_{g,1}(V\neq\square)\ll q^{\frac{g}{2}(1+\epsilon)}$. Similar calculations can be used to bound $\mathcal{S}^o_{g,2}(V\neq\square)$, $\mathcal{S}^o_{g-1,1}(V\neq\square)$ and $\mathcal{S}^o_{g-1,2}(V\neq\square)$ by $q^{\frac{g}{2}(1+\epsilon)}$. Therefore $\mathcal{S}^o(V\neq\square)\ll q^{\frac{g}{2}(1+\epsilon)}$, which proves Proposition 6.1.
\section{Proof of Theorem 1.4}
We combine the results from the previous sections to prove Theorem 1.4. \\
\begin{proof}[Proof of Theorem 1.4.]
Using (\ref{firsteq}), we have 
\begin{equation}\label{eq:7.1}
    \sum_{D\in\mathcal{H}_{2g+2}}L\left(\frac{1}{2},\chi_D\right)=\mathcal{S}_{g,1}-\mathcal{S}_{g,2}+\mathcal{S}_{g-1,1}-\mathcal{S}_{g-1,2}.
\end{equation}
Using equations stated in previous sections, we can rewrite (\ref{eq:7.1}) as 
\begin{equation}
    \sum_{D\in\mathcal{H}_{2g+2}}L\left(\frac{1}{2},\chi_D\right)=M+\mathcal{S}(V=\square)+\mathcal{S}(V\neq\square).
\end{equation}
Using Proposition 4.1, Proposition 5.1 and Proposition 6.1, we have 
\begin{align*}
    \sum_{D\in\mathcal{H}_{2g+2}}L\left(\frac{1}{2},\chi_D\right)&=M_{g,1}-M_{g,2}+M_{g-1,1}-M_{g-1,2}+\mathcal{S}_1(V=\square)+\mathcal{S}_2(V=\square)+\mathcal{S}_3(V=\square)+\mathcal{S}_4(V=\square)\\
    &+q^{\frac{2g+2}{3}}\mathcal{R}(2g+2)+q^{\frac{g}{6}+\left[\frac{g}{2}\right]}C_1+q^{\frac{g}{6}+\left[\frac{g-1}{2}\right]}C_2+O(q^{\frac{g}{2}(1+\epsilon)}).
\end{align*}
By Remark 4.3, $\mathcal{C}(u)$ has an analytic continuation for $|u|<q$ and $\mathcal{C}(1)=0$, therefore between the circles $|u|=r$ and $|u|=R$, the integrands corresponding to the terms $M_{g,1}, M_{g-1,1}, \mathcal{S}_1(V=\square)$ and $\mathcal{S}_2(V=\square)$ have a double pole at $u=q^{-1}$. Similarly the integrands corresponding to the terms $M_{g,2}, M_{g-1,2}, \mathcal{S}_3(V=\square)$ and $\mathcal{S}_4(V=\square)$ have a simple pole at $u=q^{-1}$. Computing the residue at $u=q^{-1}$, we get that
\begin{equation*}
   \mathcal{S}_1(V=\square)+M_{g,1}=\frac{q^{2g+2}}{\zeta_{\mathbb{A}}(2)}P(1)\left(\left[\frac{g}{2}\right]+1+\frac{1}{\log q}\frac{P'(1)}{P(1)}\right),
   \end{equation*}
\begin{equation*}
  \mathcal{S}_2(V=\square)+M_{g-1,1}=\frac{q^{2g+2}}{\zeta_{\mathbb{A}}(2)}P(1)\left(\left[\frac{g-1}{2}\right]+1+\frac{1}{\log q}\frac{P'(1)}{P(1)}\right),
\end{equation*}
\begin{equation*}
    \mathcal{S}_3(V=\square)-M_{g,2}=\frac{q^{\frac{3g+5}{2}+\left[\frac{g}{2}\right]}}{\zeta_{\mathbb{A}}(2)}\frac{P(1)}{(q-1)}
\end{equation*}
and
\begin{equation*}
   \mathcal{S}_4(V=\square)-M_{g-1,2}=\frac{q^{\frac{3g}{2}+3+\left[\frac{g-1}{2}\right]}}{\zeta_{\mathbb{A}}(2)}\frac{P(1)}{(q-1)}. 
\end{equation*}
where $\mathcal{C}(u)=P(s)$ with the change of variables $u=q^{-s}$. Putting everything together and using equation (\ref{eq:5.25}) the Theorem follows. 
\end{proof}
\begin{appendix}
\section{Completing the Proof of Lemma 5.6}
In the appendix we prove the claim that 
\begin{equation}\label{eq:A1}
     \mathcal{A}^o_{g-1,1,2}+\hat{\mathcal{A}}_2+\tilde{\mathcal{A}}_2-\mathcal{A}^e_{g,2,1}-\mathcal{A}^e_{g,2,2}-\mathcal{A}^e_{g-1,2,1}-\mathcal{A}^e_{g-1,2,2}-\mathcal{A}^o_{g,2}-\mathcal{A}^o_{g-1,2}
\end{equation}
equals zero. 
For the terms corresponding to the residues at $u=q^{-1}$ and $u=q^{-2}$ we have shown that (\ref{eq:A1}) equals zero, thus it remains to show that for the terms corresponding to the residue at $u=0$, (\ref{eq:A1}) equals zero. We prove this using induction on $g$.To do this we consider two cases, the first when $g$ is even and second when $g$ is odd.
\subsection{$g$ even}
Let $g=2m$ for $m\in\mathbb{Z}$, then we will show that $(\ref{eq:A1})$ equals zero for all $m\geq 1.$ For the base case, $m=1$, we have that (\ref{eq:A1}) is equal to 
\begin{align}
    \frac{1}{\zeta_{\mathbb{A}}(2)}\Bigg(&q^{\frac{9}{2}}\big(\mathcal{C}(0)+\mathcal{C}'(0)\big)+q^4\big(\mathcal{C}(0)(1+q)+\mathcal{C}'(0)\big)+q^{\frac{11}{2}}\mathcal{C}(0)+q^{\frac{13}{2}}\mathcal{C}(0)-q^{\frac{11}{2}}\mathcal{C}(0)\nonumber\\&+q^5\big(\mathcal{C}(0)(q+q^2)+\mathcal{C}'(0)\big)-q^4\big(\mathcal{C}(0)(1+q^2)+\mathcal{C}'(0)\big)-q^{\frac{9}{2}}\big(\mathcal{C}(0)(1+q^2)+\mathcal{C}'(0)\big)\nonumber\\&-q^{5}\big(\mathcal{C}(0)(1+q^2)+\mathcal{C}'(0)\big)\Bigg),
\end{align}
which when cancelling the terms equals zero, hence the base case is true. Assume that (\ref{eq:A1})$=0$, for $m=t$. Then we have that 
\begin{align}
    \frac{1}{\zeta_{\mathbb{A}}(2)}\Bigg(&q^{3t+\frac{3}{2}}\sum_{n=0}^t\frac{\mathcal{C}^{(n)}(0)}{n!}+q^{3t+1}\sum_{n=0}^t\frac{\mathcal{C}^{(n)}(0)}{n!}\sum_{k=0}^{t-n}q^k+q^{3t+\frac{5}{2}}\sum_{n=0}^{t-1}\frac{\mathcal{C}^{(n)}(0)}{n!}\sum_{k=0}^{t-1-n}q^k\nonumber\\&+q^{3t+\frac{7}{2}}\sum_{n=0}^{t-1}\frac{\mathcal{C}^{(n)}(0)}{n!}\sum_{k=t-1-n}^{2(t-1-n)}q^k-q^{3t+\frac{5}{2}}\sum_{n=0}^{t-1}\frac{\mathcal{C}^{(n)}(0)}{n!}\sum_{k=0}^{t-1-n}q^{2k}+q^{3t+2}\sum_{n=0}^t\frac{\mathcal{C}^{(n)}(0)}{n!}\sum_{k=t-n}^{2(t-n)}q^k\nonumber\\&-q^{3t+1}\sum_{n=0}^t\frac{\mathcal{C}^{(n)}(0)}{n!}\sum_{k=0}^{t-n}q^{2k}-q^{3t+\frac{3}{2}}\sum_{n=0}^t\frac{\mathcal{C}^{(n)}(0)}{n!}\sum_{k=0}^{t-n}q^{2k}-q^{3t+2}\sum_{n=0}^t\frac{\mathcal{C}^{(n)}(0)}{n!}\sum_{k=0}^{t-n}q^{2k}\Bigg)=0.
\end{align}
It remains to show that (\ref{eq:A1}) is equal to zero for $m=t+1$. For $m=t+1$, we have that (\ref{eq:A1}) is equal to 
\begin{align}\label{eq:A4}
    \frac{1}{\zeta_{\mathbb{A}}(2)}\Bigg(&q^{3t+\frac{9}{2}}\sum_{n=0}^{t+1}\frac{\mathcal{C}^{(n)}(0)}{n!}+q^{3t+4}\sum_{n=0}^{t+1}\frac{\mathcal{C}^{(n)}(0)}{n!}\sum_{k=0}^{t+1-n}q^k+q^{3t+\frac{11}{2}}\sum_{n=0}^t\frac{\mathcal{C}^{(n)}(0)}{n!}\sum_{k=0}^{t-n}q^k\nonumber\\
    &+q^{3t+\frac{13}{2}}\sum_{n=0}^t\frac{\mathcal{C}^{(n)}(0)}{n!}\sum_{k=t-n}^{2(t-n)}q^k-q^{3t+\frac{11}{2}}\sum_{n=0}^t\frac{\mathcal{C}^{(n)}(0)}{n!}\sum_{k=0}^{t-n}q^{2k}+q^{3t+5}\sum_{n=0}^{t+1}\frac{\mathcal{C}^{(n)}(0)}{n!}\sum_{k=t+1-n}^{2(t+1-n)}q^k\nonumber\\
    &-q^{3t+4}\sum_{n=0}^{t+1}\frac{\mathcal{C}^{(n)}(0)}{n!}\sum_{k=0}^{t+1-n}q^{2k}-q^{3t+\frac{9}{2}}\sum_{n=0}^{t+1}\frac{\mathcal{C}^{(n)}(0)}{n!}\sum_{k=0}^{t+1-n}q^{2k}-q^{3t+5}\sum_{n=0}^{t+1}\frac{\mathcal{C}^{(n)}(0)}{n!}\sum_{k=0}^{t+1-n}q^{2k}\Bigg).
\end{align}
Rearranging (\ref{eq:A4}), we have that (\ref{eq:A1}) is equal to 
\begin{align}\label{eq:A5}
    \frac{q^3}{\zeta_{\mathbb{A}}(2)}\Bigg(&q^{3t+\frac{3}{2}}\sum_{n=0}^t\frac{\mathcal{C}^{(n)}(0)}{n!}+q^{3t+1}\sum_{n=0}^t\frac{\mathcal{C}^{(n)}(0)}{n!}\sum_{k=0}^{t-n}q^k+q^{3t+\frac{5}{2}}\sum_{n=0}^{t-1}\frac{\mathcal{C}^{(n)}(0)}{n!}\sum_{k=0}^{t-1-n}q^k\nonumber\\&+q^{3t+\frac{7}{2}}\sum_{n=0}^{t-1}\frac{\mathcal{C}^{(n)}(0)}{n!}\sum_{k=t-1-n}^{2(t-1-n)}q^k-q^{3t+\frac{5}{2}}\sum_{n=0}^{t-1}\frac{\mathcal{C}^{(n)}(0)}{n!}\sum_{k=0}^{t-1-n}q^{2k}+q^{3t+2}\sum_{n=0}^t\frac{\mathcal{C}^{(n)}(0)}{n!}\sum_{k=t-n}^{2(t-n)}q^k\nonumber\\&-q^{3t+1}\sum_{n=0}^t\frac{\mathcal{C}^{(n)}(0)}{n!}\sum_{k=0}^{t-n}q^{2k}-q^{3t+\frac{3}{2}}\sum_{n=0}^t\frac{\mathcal{C}^{(n)}(0)}{n!}\sum_{k=0}^{t-n}q^{2k}-q^{3t+2}\sum_{n=0}^t\frac{\mathcal{C}^{(n)}(0)}{n!}\sum_{k=0}^{t-n}q^{2k}\Bigg)
    \end{align}
    \begin{align}
    +\frac{1}{\zeta_{\mathbb{A}}(2)}\Bigg(&q^{3t+\frac{9}{2}}\frac{\mathcal{C}^{(t+1)}(0)}{(t+1)!}+q^{3t+4}\sum_{n=0}^{t+1}\frac{\mathcal{C}^{(n)}(0)}{n!}q^{t+1-n}+q^{3t+\frac{11}{2}}\sum_{n=0}^t\frac{\mathcal{C}^{(n)}(0)}{n!}q^{t-n}-q^{3t+\frac{13}{2}}\sum_{n=0}^{t-1}\frac{\mathcal{C}^{(n)}(0)}{n!}q^{t-1-n}\nonumber\\&+q^{3t+\frac{13}{2}}\sum_{n=0}^{t-1}\frac{\mathcal{C}^{(n)}(0)}{n!}q^{2(t-n)-1}(1+q)+q^{3t+\frac{13}{2}}\frac{\mathcal{C}^{(t)}(0)}{t!}-q^{3t+\frac{11}{2}}\sum_{n=0}^t\frac{\mathcal{C}^{(n)}(0)}{n!}q^{2(t-n)}\nonumber\\
    &-q^{3t+5}\sum_{n=0}^t\frac{\mathcal{C}^{(n)}(0)}{n!}q^{t-n}+q^{3t+5}\sum_{n=0}^t\frac{\mathcal{C}^{(n)}(0)}{n!}q^{2(t-n)+1}(1+q)+q^{3t+5}\frac{\mathcal{C}^{(t+1)}(0)}{(t+1)!}\nonumber\\&-q^{3t+4}\sum_{n=0}^{t+1}\frac{\mathcal{C}^{(n)}(0)}{n!}q^{2(t+1-n)}-q^{3t+\frac{9}{2}}\sum_{n=0}^{t+1}\frac{\mathcal{C}^{(n)}(0)}{n!}q^{2(t+1-n)}-q^{3t+5}\sum_{n=0}^{t+1}\frac{\mathcal{C}^{(n)}(0)}{n!}q^{2(t+1-n)}\Bigg)\label{eq:A6}.
\end{align}
Using the inductive hypothesis, we have that (\ref{eq:A5}) equals zero, therefore it remains to show that (\ref{eq:A6}) equals zero. Rearranging (\ref{eq:A6}) we get that it is equal to
\begin{align*}
    \frac{1}{\zeta_{\mathbb{A}}(2)}\Bigg(&-q^{3t+\frac{13}{2}}\sum_{n=0}^t\frac{\mathcal{C}^{(n)}(0)}{n!}q^{2(t-n)}+q^{3t+4}\frac{\mathcal{C}^{(t+1)}(0)}{(t+1)!}+q^{3t+\frac{11}{2}}\frac{\mathcal{C}^{(t)}(0)}{t!}+q^{3t+\frac{11}{2}}\sum_{n=0}^{t-1}\frac{\mathcal{C}^{(n)}(0)}{n!}q^{2(t-n)}\\&+q^{3t+\frac{13}{2}}\sum_{n=0}^{t-1}\frac{\mathcal{C}^{(n)}(0)}{n!}q^{2(t-n)}+q^{3t+\frac{13}{2}}\frac{\mathcal{C}^{(t)}(0)}{t!}-q^{3t+\frac{11}{2}}\sum_{n=0}^t\frac{\mathcal{C}^{(n)}(0)}{n!}q^{2(t-n)}+q^{3t+6}\sum_{n=0}^t\frac{\mathcal{C}^{(n)}(0)}{n!}q^{2(t-n)}\\&+q^{3t+7}\sum_{n=0}^t\frac{\mathcal{C}^{(n)}(0)}{n!}q^{2(t-n)}-q^{3t+7}\sum_{n=0}^t\frac{\mathcal{C}^{(n)}(0)}{n!}q^{2(t-n)}-q^{3t+4}\sum_{n=0}^{t+1}\frac{\mathcal{C}^{(n)}(0)}{n!}q^{2(t+1-n)}\Bigg)\\
    =\frac{1}{\zeta_{\mathbb{A}}(2)}\Bigg(&q^{3t+\frac{11}{2}}(1+q^{\frac{1}{2}}+q+q^{\frac{3}{2}})\sum_{n=0}^t\frac{\mathcal{C}^{(n)}(0)}{n!}q^{2(t-n)}-q^{3t+\frac{11}{2}}(1+q^{\frac{1}{2}}+q+q^{\frac{3}{2}})\sum_{n=0}^t\frac{\mathcal{C}^{(n)}(0)}{n!}q^{2(t-n)}\Bigg)=0.
\end{align*}
Thus (\ref{eq:A1})$=0$ for $m=t+1$ and so, by induction, (\ref{eq:A1})$=0$ for all $g\geq 1, g$ even.
\subsection{$g$ odd}
Now let $g=2m+1$, then we want to show, using induction on $m$ that (\ref{eq:A1}) equals zero for all $m\geq 0.$ For the base case, $m=0$, we have that (\ref{eq:A1}) is equal to
\begin{align*}
    \frac{1}{\zeta_{\mathbb{A}}(2)}\Bigg(&q^{\frac{5}{2}}\big(\mathcal{C}(0)+\mathcal{C}'(0)\big)+q^{\frac{7}{2}}\mathcal{C}(0)+q^3\mathcal{C}(0)+q^4\mathcal{C}(0)-q^3\mathcal{C}(0)+q^\frac{9}{2}\mathcal{C}(0)-q^{\frac{7}{2}}\mathcal{C}(0)\\&-q^4\mathcal{C}(0)-q^{\frac{5}{2}}\big(\mathcal{C}(0)(1+q^2)+\mathcal{C}'(0)\big)\Bigg),
\end{align*}
which, when cancelling the terms equals zero, hence the base case is true. Assume that (\ref{eq:A1})$=0$ for $m=t$. Then we have that
\begin{align}
    \frac{1}{\zeta_{\mathbb{A}}(2)}\Bigg(&q^{3t+\frac{5}{2}}\sum_{n=0}^{t+1}\frac{\mathcal{C}^{(n)}(0)}{n!}+q^{3t+\frac{7}{2}}\sum_{n=0}^t\frac{\mathcal{C}^{(n)}(0)}{n!}\sum_{k=0}^{t-n}q^k+q^{3t+3}\sum_{n=0}^t\frac{\mathcal{C}^{(n)}(0)}{n!}\sum_{k=0}^{t-n}q^k\nonumber\\
    &+q^{3t+4}\sum_{n=0}^t\frac{\mathcal{C}^{(n)}(0)}{n!}\sum_{k=t-n}^{2(t-n)}q^k-q^{3t+3}\sum_{n=0}^t\frac{\mathcal{C}^{(n)}(0)}{n!}\sum_{k=0}^{t-n}q^{2k}+q^{3t+\frac{9}{2}}\sum_{n=0}^t\frac{\mathcal{C}^{(n)}(0)}{n!}\sum_{k=t-n}^{2(t-n)}q^k\nonumber\\&-q^{3t+\frac{7}{2}}\sum_{n=0}^t\frac{\mathcal{C}^{(n)}(0)}{n!}\sum_{k=0}^{t-n}q^{2k}-q^{3t+4}\sum_{n=0}^t\frac{\mathcal{C}^{(n)}(0)}{n!}\sum_{k=0}^{t-n}q^{2k}-q^{3t+\frac{5}{2}}\sum_{n=0}^{t+1}\frac{\mathcal{C}^{(n)}(0)}{n!}\sum_{k=0}^{t+1-n}q^{2k}\Bigg)=0.
\end{align}
It remains to show that (\ref{eq:A1}) is equal to zero for $m=t+1$. For $m=t+1$, we have that (\ref{eq:A1}) is equal to 
\begin{align}\label{eq:A8}
    \frac{1}{\zeta_{\mathbb{A}}(2)}\Bigg(&q^{3t+\frac{11}{2}}\sum_{n=0}^{t+2}\frac{\mathcal{C}^{(n)}(0)}{n!}+q^{3t+\frac{13}{2}}\sum_{n=0}^{t+1}\frac{\mathcal{C}^{(n)}(0)}{n!}\sum_{k=0}^{t+1-n}q^k+q^{3t+6}\sum_{n=0}^{t+1}\frac{\mathcal{C}^{(n)}(0)}{n!}\sum_{k=0}^{t+1-n}q^k\nonumber\\
    &+q^{3t+7}\sum_{n=0}^{t+1}\frac{\mathcal{C}^{(n)}(0)}{n!}\sum_{k=t+1-n}^{2(t+1-n)}q^k-q^{3t+6}\sum_{n=0}^{t+1}\frac{\mathcal{C}^{(n)}(0)}{n!}\sum_{k=0}^{t+1-n}q^{2k}+q^{3t+\frac{15}{2}}\sum_{n=0}^{t+1}\frac{\mathcal{C}^{(n)}(0)}{n!}\sum_{k=t+1-n}^{2(t+1-n)}q^k\nonumber\\
    &-q^{3t+\frac{13}{2}}\sum_{n=0}^{t+1}\frac{\mathcal{C}^{(n)}(0)}{n!}\sum_{k=0}^{t+1-n}q^{2k}-q^{3t+7}\sum_{n=0}^{t+1}\frac{\mathcal{C}^{(n)}(0)}{n!}\sum_{k=0}^{t+1-n}q^{2k}-q^{3t+\frac{11}{2}}\sum_{n=0}^{t+2}\frac{\mathcal{C}^{(n)}(0)}{n!}\sum_{k=0}^{t+2-n}q^{2k}\Bigg).
\end{align}
Rearranging (\ref{eq:A8}), we have that (\ref{eq:A1}) is equal to 
\begin{align}
    \frac{q^3}{\zeta_{\mathbb{A}}(2)}\Bigg( &q^{3t+\frac{5}{2}}\sum_{n=0}^{t+1}\frac{\mathcal{C}^{(n)}(0)}{n!}+q^{3t+\frac{7}{2}}\sum_{n=0}^t\frac{\mathcal{C}^{(n)}(0)}{n!}\sum_{k=0}^{t-n}q^k+q^{3t+3}\sum_{n=0}^t\frac{\mathcal{C}^{(n)}(0)}{n!}\sum_{k=0}^{t-n}q^k\nonumber\\
    &+q^{3t+4}\sum_{n=0}^t\frac{\mathcal{C}^{(n)}(0)}{n!}\sum_{k=t-n}^{2(t-n)}q^k-q^{3t+3}\sum_{n=0}^t\frac{\mathcal{C}^{(n)}(0)}{n!}\sum_{k=0}^{t-n}q^{2k}+q^{3t+\frac{9}{2}}\sum_{n=0}^t\frac{\mathcal{C}^{(n)}(0)}{n!}\sum_{k=t-n}^{2(t-n)}q^k\nonumber\\&-q^{3t+\frac{7}{2}}\sum_{n=0}^t\frac{\mathcal{C}^{(n)}(0)}{n!}\sum_{k=0}^{t-n}q^{2k}-q^{3t+4}\sum_{n=0}^t\frac{\mathcal{C}^{(n)}(0)}{n!}\sum_{k=0}^{t-n}q^{2k}-q^{3t+\frac{5}{2}}\sum_{n=0}^{t+1}\frac{\mathcal{C}^{(n)}(0)}{n!}\sum_{k=0}^{t-n}q^{2k}\Bigg)\label{eq:A9}\\
   +\frac{1}{\zeta_{\mathbb{A}}(2)}\Bigg(&q^{3t+\frac{11}{2}}\frac{\mathcal{C}^{(t+2)}(0)}{(t+2)!}+q^{3t+\frac{13}{2}}\sum_{n=0}^{t+1}\frac{\mathcal{C}^{(n)}(0)}{n!}q^{t+1-n}+q^{3t+6}\sum_{n=0}^{t+1}\frac{\mathcal{C}^{(n)}(0)}{n!}q^{t+1-n}-q^{3t+7}\sum_{n=0}^t\frac{\mathcal{C}^{(n)}(0)}{n!}q^{t-n}\nonumber\\
   &+q^{3t+7}\sum_{n=0}^t\frac{\mathcal{C}^{(n)}(0)}{n!}q^{2(t+1-n)-1}(1+q)+q^{3t+7}\frac{\mathcal{C}^{(t+1)}(0)}{(t+1)!}-q^{3t+6}\sum_{n=0}^{t+1}\frac{\mathcal{C}^{(n)}(0)}{n!}q^{2(t+1-n)}\nonumber\\
   &-q^{3t+\frac{15}{2}}\sum_{n=0}^t\frac{\mathcal{C}^{(n)}(0)}{n!}q^{t-n}+q^{3t+\frac{15}{2}}\sum_{n=0}^t\frac{\mathcal{C}^{(n)}(0)}{n!}q^{2(t+1-n)-1}(1+q)+q^{3t+\frac{15}{2}}\frac{\mathcal{C}^{(t+1)}(0)}{(t+1)!}\nonumber\\
   &-q^{3t+\frac{13}{2}}\sum_{n=0}^{t+1}\frac{\mathcal{C}^{(n)}(0)}{n!}q^{2(t+1-n)}-q^{3t+7}\sum_{n=0}^{t+1}\frac{\mathcal{C}^{(n)}(0)}{n!}q^{2(t+1-n)}-q^{3t+\frac{11}{2}}\sum_{n=0}^{t+2}\frac{\mathcal{C}^{(n)}(0)}{n!}q^{2(t+2-n)}\Bigg).\label{eq:A10}
\end{align}
Using the inductive hypothesis, we have (\ref{eq:A9}) equals zero. Thus it remains to show that (\ref{eq:A10}) equals zero. Rearranging (\ref{eq:A10}), we have that it equals
\begin{align*}
    \frac{1}{\zeta_{\mathbb{A}}(2)}\Bigg(&-q^{3t+\frac{15}{2}}\sum_{n=0}^{t+1}\frac{\mathcal{C}^{(n)}(0)}{n!}q^{2(t+1-n)}+q^{3t+\frac{13}{2}}\frac{\mathcal{C}^{(t+1)}(0)}{(t+1)!}+q^{3t+6}\frac{\mathcal{C}^{(t+1)}(0)}{(t+1)!}+q^{3t+6}\sum_{n=0}^t\frac{\mathcal{C}^{(n)}(0)}{n!}q^{2(t+1-n)}\\
    &+q^{3t+7}\sum_{n=0}^t\frac{\mathcal{C}^{(n)}(0)}{n!}q^{2(t+1-n)}+q^{3t+7}\frac{\mathcal{C}^{(t+1)}(0)}{(t+1)!}-q^{3t+6}\sum_{n=0}^{t+1}\frac{\mathcal{C}^{(n)}(0)}{n!}q^{2(t+1-n)}\\
    &+q^{3t+\frac{13}{2}}\sum_{n=0}^t\frac{\mathcal{C}^{(n)}(0)}{n!}q^{2(t+1-n)}+q^{3t+\frac{15}{2}}\sum_{n=0}^t\frac{\mathcal{C}^{(n)}(0)}{n!}q^{2(t+1-n)}+q^{3t+\frac{15}{2}}\frac{\mathcal{C}^{(t+1)}(0)}{(t+1)!}\\
    &-q^{3t+\frac{13}{2}}\sum_{n=0}^{t+1}\frac{\mathcal{C}^{(n)}(0)}{n!}q^{2(t+1-n)}-q^{3t+7}\sum_{n=0}^{t+1}\frac{\mathcal{C}^{(n)}(0)}{n!}q^{2(t+1-n)}\Bigg)\\
    =\frac{1}{\zeta_{\mathbb{A}}(2)}\Bigg(&q^{3t+6}(1+q^{\frac{1}{2}}+q+q^{\frac{3}{2}})\sum_{n=0}^{t+1}\frac{\mathcal{C}^{(n)}(0)}{n!}q^{2(t+1-n)}-q^{3t+6}(1+q^{\frac{1}{2}}+q+q^{\frac{3}{2}})\sum_{n=0}^{t+1}\frac{\mathcal{C}^{(n)}(0)}{n!}q^{2(t+1-n)}\bigg)=0.
\end{align*}
Thus (\ref{eq:A1})$=0$ for $m=t+1$ and so by induction (\ref{eq:A1})$=0$ for all $g\geq 1, g$ odd. This completes the proof of Lemma 5.6.\\
\par\noindent
\textbf{Acknowledgement:} The first author is grateful to the Leverhulme Trust (RPG-2017-320) for the support through the research project grant "Moments of L-functions in Function Fields and Random Matrix Theory". The second author is also grateful to the Leverhulme Trust (RPG-2017-320) for the support given during this research through a PhD studentship.
\end{appendix}
\bibliographystyle{plain}
\bibliography{realquadfuncfield}
Department of Mathematics, University of Exeter, Exeter, EX4 4QF, UK\\
\textit{E-mail Address:}j.c.andrade@exeter.ac.uk\\
\textit{E-mail Address:}jm1015@exeter.ac.uk
\end{document}